\documentclass[12pt,oneside]{amsart}
\usepackage[utf8]{inputenc}
\usepackage{amsmath}
\usepackage{tikz-cd}
\usepackage{bm}
\usepackage{amssymb}
\usepackage{mathtools}
\usepackage{amsthm}
\usepackage{graphicx}
\usepackage{comment}
\usepackage{bbold}
\usepackage{comment}
\usepackage{appendix}
\usepackage{caption}
\usepackage{subcaption}
\usepackage{hyperref}
\usepackage{upgreek}
\usepackage{enumitem}
\usepackage{comment}
\usepackage{tikz}
\usepackage{pict2e,mathrsfs}

\usepackage[margin=3cm]{geometry}

\title{Reflexivity and Hochschild Cohomology}
\author{Isambard Goodbody}
\date{\today}

\DeclareUnicodeCharacter{0327}{\c{}}

\DeclareMathOperator{\iso}{iso}

\DeclareMathOperator{\Ho}{Ho}

\DeclareMathOperator{\Res}{Res}
\DeclareMathOperator{\ev}{ev}
\DeclareMathOperator{\eval}{eval}

\DeclareMathOperator{\oHom}{Hom}

\DeclareMathOperator{\RHom}{\underline{\mathbb{R}Hom}}

\DeclareMathOperator{\thick}{thick}

\DeclareMathOperator{\Mod}{Mod}

\DeclareMathOperator{\coh}{coh}
\DeclareMathOperator{\perf}{perf}

\DeclareMathOperator{\fd}{fd}

\DeclareMathOperator{\cf}{fd}

\DeclareMathOperator{\dgcat}{DGcat}

\DeclareMathOperator{\Hqe}{Hqe}
\DeclareMathOperator{\Hmo}{Hmo}

\DeclareMathOperator{\Refl}{Refl}
\DeclareMathOperator{\Dual}{Dual}

\DeclareMathOperator{\Hom}{\underline{Hom}}
\DeclareMathOperator{\uD}{\underline{\mathcal{D}}}
\DeclareMathOperator{\D}{\mathcal{D}}
\DeclareMathOperator{\uC}{\underline{\mathcal{C}}}

\DeclareMathOperator{\DPic}{DPic}

\begin{document}

\usetikzlibrary{matrix}
\usetikzlibrary{shapes}

\theoremstyle{plain}
\newtheorem{prop}{Proposition}[section]
\newtheorem{lemma}[prop]{Lemma}
\newtheorem{theorem}[prop]{Theorem}
\newtheorem{cor}[prop]{Corollary}
\newtheorem{conj}[prop]{Conjecture}

\theoremstyle{definition}
\newtheorem{defn}[prop]{Definition}
\newtheorem{ass}[prop]{Assumption}
\newtheorem{cons}[prop]{Construction}
\newtheorem{ex}[prop]{Example}
\newtheorem{remark}[prop]{Remark}
\newtheorem*{ack}{Acknowledgements}
\maketitle

\tikzset{
    vert/.style={anchor=south, rotate=90, inner sep=.5mm}
}

\newcommand{\rightarrowdbl}{\rightarrow\mathrel{\mkern-14mu}\rightarrow}

\newcommand{\xrightarrowdbl}[2][]{%
  \xrightarrow[#1]{#2}\mathrel{\mkern-14mu}\rightarrow
}

\begin{abstract}
We characterise reflexive DG-categories, as introduced by Kuznetsov and Shinder, as the reflexive objects in the closed symmetric monoidal category of DG-categories localised at Morita equivalences. As consequences, we show that the Hochschild cohomology and the derived Picard group of a reflexive DG-category coincide with those of its derived category of cohomologically finite modules.

\end{abstract}

\tableofcontents

\section{Introduction}

Reflexivity is about a duality between two types of derived categories which appear in algebra and geometry. If $X$ is a projective scheme, then the difference between the category of perfect complexes $\mathcal{D}^{\perf}(X)$ and the bounded derived category $\mathcal{D}^b(\coh X)$ is a measure of smoothness; they coincide if and only if $X$ is regular. Similarly if $A$ is a finite-dimensional algebra, then $A$ has finite global dimension if and only if its perfect and bounded derived categories coincide. Reflexivity concerns the interplay of these categories in the singular setting -- when the two categeories are distinct.
\\

Results indicating that these two triangulated categories determine each other and satisfy a kind of duality have been in the air for a long time. For example, see \cite[Theorem A.1]{VB02}, \cite[Corollary 8.3]{Ric89}, \cite[Theorems 3.2, 3.12]{Bal11} and \cite[Theorem 2.8]{Che21}.  Reflexive DG-categories were defined in \cite{KS22} to abstract this duality in an enhanced setting. That is to say, the DG-structure is taken into account, instead of just the triangulated structure. The key motivating examples of reflexive DG-categories are the perfect and bounded derived categories of projective schemes and finite-dimensional algebras. For a survey of other examples, see \cite{BGO25}.
\\

The two DG-categories associated to a reflexive DG-category $\mathcal{A}$ which abstract the perfect and bounded derived categories are the perfect derived category $\uD^{\perf}(\mathcal{A})$ and the category of cohomologically finite modules $\uD_{\fd}(\mathcal{A})$ (see Section \ref{DGprelims} for definitions). These two categories share a great deal of common information; it was shown in \cite[Corollaries 3.16 and 3.17]{KS22} that there is a bijection between their classes of semiorthogonal decompositions and that their
triangulated autoequivalence groups are isomorphic. This information lives at the triangulated level and in this paper, we show that there is also a significant amount of
enhanced information shared between these two categories. 
\\

The Hochschild cohomology of DG-category is a powerful invariant which controls its deformation theory and has a geometric interpretation via the Hochschild-Konstant-Rosenberg isomorphism (see Section \ref{HHsectionfd}). Semireflexivity (Definition \ref{reflsemirefldefn}) is a weaker property than reflexivity and is satisfied by any proper DG-category by \cite[Lemma 3.14]{KS22}. 

\begin{theorem}[Theorem \ref{semireflhhiso}] \label{introhhiso}
If $\mathcal{A}$ is a semireflexive DG-category, there is a quasi-isomorphism between the Hochschild cohomology DG-algebras. 
\[
HH(\uD^{\perf}(\mathcal{A})) \simeq HH(\uD_{\fd}(\mathcal{A}))
\]
\end{theorem}

The derived Picard group of an enhanced triangulated category is its group of enhanced autoequivalences. Theorem \ref{introdpic} can be thought of as an enhanced version of \cite[Corollary 3.16]{KS22}. 

\begin{theorem}[Theorem \ref{DPic}] \label{introdpic}
If $\mathcal{A}$ is reflexive DG-category, there is an isomorphism of derived Picard groups 
\[
\DPic(\uD^{\perf}(\mathcal{A})) \simeq \DPic(\uD_{\fd}(\mathcal{A}))
\]
\end{theorem}

Theorems \ref{introhhiso} and \ref{introdpic} follow from a new characterisation reflexivity in terms of the monoidal structure on the category $\Hmo$ of DG-categories localised at Morita equivalences (Section \ref{duals}). By  \cite[Theorem 6.1]{toe06}, the  monoidal category $\Hmo$ is closed i.\,e. it has an internal hom. In any closed symmetric monoidal category, the reflexive objects are defined as the largest subcategory on which the duality functor restricts to an equivalence. 

\begin{theorem}[Theorem \ref{samedef}] \label{samedefintro}
A DG-category is reflexive if and only if it represents a reflexive object in the closed symmetric monoidal category $\Hmo$.
\end{theorem}

In a closed symmetric monoidal category, any dualisable object is reflexive but the converse is not true. By \cite[Theorem 1.43]{Tab15}, the dualisable objects in $\Hmo$ are the smooth and proper DG-categories. This includes derived categories of smooth projective schemes and finite-dimensional algebras of finite global dimension over perfect fields. Theorem \ref{samedefintro} provides a monoidal comparison between reflexive DG-categories and smooth and proper DG-categories.
\\

In Section \ref{prelims}, we provide some preliminaries on monoidal categories and DG-categories. In Section \ref{KSreflsect}, we recall Kuznetsov and Shinder's definition of a reflexive DG-category along with a detailed construction of the evaluation map involved which is omitted from \cite{KS22}. In Section \ref{reflDGsec}, we prove Theorem \ref{samedefintro}. In Section \ref{HHreflsec}, we apply this to prove Theorems \ref{introhhiso} and \ref{introdpic}. In Section \ref{fdsec}, we consider the special case of finite-dimensional DG-algebras and show that the power series ring is reflexive (Example \ref{powerseriesrefl}).

\begin{ack}
I'd like to thank my supervisor Greg Stevenson for his helpful guidance. I am also grateful to Alexander Kuznetsov for useful comments and an anonymous reviewer for a number of suggestions. I am supported by a PhD scholarship from the Carnegie Trust for the Universities of Scotland. 
\end{ack}

\section{Preliminaries} \label{prelims}

\subsection{Monoidal Categories}
\label{sectreflmon}

We collect some well-known results on monoidal categories. See \cite{DP84}, \cite[Section III.1]{LMS86}, or \cite{EGNO15}, for more background. Let $(\mathcal{M},\otimes,\bm{1})$ be a symmetric monoidal category whose internal hom we denote $\hom(-,-)$. Recall this means that there is a functor 
\[
- \otimes -: \mathcal{M} \times \mathcal{M} \to \mathcal{M}
\]
which behaves like the tensor product of vector spaces in that it is associative, symmetric and unital with respect to $\bm{1} \in \mathcal{M}$, up to some coherent isomorphisms. We will furthermore assume that $\mathcal{M}$ is closed which means there is an internal hom functor 
\[
\hom(-,-): \mathcal{M}^{op} \times \mathcal{M} \to \mathcal{M} 
\]
which satisfies the tensor-hom adjunction. We define the duality functor as 
\[
D \coloneq \hom(-,\bm{1}) \colon \mathcal{M} \to \mathcal{M}^{op}.
\]
It is left adjoint to $D^{op} \colon \mathcal{M}^{op} \to \mathcal{M}$. Indeed, there are natural isomorphisms
\begin{equation}\label{internaladjunction}
\hom(x,Dy) \simeq \hom(x \otimes y, \bm{1}) \simeq \hom(y \otimes x, \bm{1}) \simeq \hom(y,Dx).  
\end{equation}
The unit and the counit coincide as a natural transformation
\[
\eval \colon 1_{\mathcal{M}} \to DD.
\] 
\begin{defn}
An object $x \in \mathcal{M}$ is reflexive if $\eval_x$ is an isomorphism. Let $\Refl(\mathcal{M}) \subseteq \mathcal{M}$ denote the full subcategory of reflexive objects of $\mathcal{M}$.
\end{defn}

\begin{prop} \label{reflexivechar}
If $\mathcal{M}$ is a closed symmetric monoidal category, then $D$ restricts to a self-inverse equivalence 
\[
D\colon \Refl(\mathcal{M}) \xrightarrow{\sim}\Refl(\mathcal{M})^{op} 
\]
and if $D$ restricts to an equivalence $D: \mathcal{C} \xrightarrow{\sim} \mathcal{C}^{op}$ for a subcategory $\mathcal{C} \subseteq \mathcal{M}$, then $\mathcal{C} \subseteq \Refl(\mathcal{M})$. 
\end{prop}

\begin{proof}
The triangle identity of the adjunction implies that $D$ restricts to a contravariant self-adjunction on $\Refl(\mathcal{M})$. Since the unit and counit are isomorphisms for all reflexive objects, it is an equivalence. If $D$ restricts to an equivalence on $\mathcal{C}$, then the unit and counit must be isomorphisms and so every object of $\mathcal{C}$ is reflexive. 
\end{proof}

\begin{remark} We consider some canonical morphisms in $\mathcal{M}$. 

\begin{enumerate} \label{laxmonoidalfunc}
\item For $x \in \mathcal{M}$, the counit of the adjunction $- \otimes x \dashv \hom(x,-)$ evaluated at the unit produces a map $\varepsilon_{\bm{1}}^x\colon Dx  \otimes x \to \bm{1}$. 

\item For $x,y \in \mathcal{M}$, the adjunct of $\varepsilon_{\bm{1}}^x \otimes \varepsilon_{\bm{1}}^y$ produces maps 
\[
\mu_{x,y}\colon Dx \otimes Dy \to D(x \otimes y)
\]
which make  $D\colon \mathcal{M} \to \mathcal{M}^{op}$ a lax monoidal functor (recall a functor $F$ between monoidal categories is lax monoidal if there are coherent morphisms $F(x) \otimes F(y) \to F(x \otimes y)$ for all objects $x,y$).

\item For ${x,y \in \mathcal{M}}$, there is a map 
\[
Dx \otimes y \to \hom(x,y)
\]
defined as the image of $\varepsilon_{\bm{1}}^x$ under
\[
\oHom(Dx \otimes x, \bm{1}) \xrightarrow{- \otimes y} \oHom(Dx \otimes x \otimes y, y) \simeq \oHom(Dx \otimes y, \hom(x,y)).
\]

\end{enumerate}
\end{remark}

Reflexivity should be compared to dualisability.

\begin{defn}
An object in $x \in \mathcal{M}$ is dualisable if the canonical map 
\[
Dx \otimes x \to \hom(x,x) 
\]
is an isomorphism. Let $\Dual(\mathcal{M})$ denote the full subcategory of dualisable objects in $\mathcal{M}$.  
\end{defn}

There are various equivalent definitions of dualisability. The following proposition can be extracted from \cite[Theorem 1.3]{DP84} and \cite[Proposition III.1.3]{LMS86}.

\begin{prop} \label{dualdef}
Let $\mathcal{M}$ be a closed symmetric monoidal category and $x \in \mathcal{M}$. The following are equivalent.
\begin{enumerate}
\item $x$ is dualisable.

\item There is a map $c\colon \bm{1} \to x \otimes Dx$ such that the composites
\begin{align*}
 x \xrightarrow{c \otimes x} & x \otimes Dx \otimes x \xrightarrow{x \otimes \varepsilon_{\bm{1}}}  x \\
Dx \xrightarrow{Dx \otimes c} & Dx \otimes x \otimes Dx  \xrightarrow{\varepsilon_{\bm{1}} \otimes Dx }  Dx 
\end{align*}
are both the identity.

\item The canonical map $Dx \otimes y \to \hom(x,y)$ is an isomorphism for all $y \in \mathcal{M}$.

\item $x$ is reflexive and $\mu_{x,y}\colon Dx \otimes Dy \to D(x \otimes y)$ is an isomorphism for all $y \in \mathcal{M}$.

\item $x$ is reflexive and $\mu_{x,Dx}\colon Dx \otimes DDx \to D(x \otimes Dx)$ is an isomorphism. 

\item $x$ is reflexive and $Dx$ is dualisable. 

\end{enumerate}
\end{prop}

\begin{prop}
Suppose $\mathcal{M}$ is a closed symmetric monoidal category. Then $\Dual(\mathcal{M})$ is a symmetric monoidal subcategory of $\mathcal{M}$ and $D$ restricts to a strong monoidal equivalence 
\[
D\colon \Dual(\mathcal{M}) \xrightarrow{\sim} \Dual(\mathcal{M})^{op}.
\]
and $\Dual(\mathcal{M})$ is maximal with respect to this property. 
\end{prop}

\begin{proof}
If $x$ and $y$ are dualisable, then by (3) and (4) of Proposition \ref{dualdef}, 
\begin{align*}
D(x \otimes y) \otimes x \otimes y & \simeq Dx \otimes Dy \otimes x \otimes y  \\
& \simeq \hom(x,Dy \otimes x \otimes y) \\ 
& \simeq \hom(x,\hom(y,x \otimes y)) \\
& \simeq \hom(x \otimes y, x \otimes y).
\end{align*}
One can check this is the canonical map and so $x \otimes y$ is dualisable. Recall by Remark \ref{laxmonoidalfunc}, $D$ is a lax monoidal functor via the morphisms $\mu_{x,y}$. By definition, $D$ is strong monoidal exactly when these are isomorphisms. So by (4), $D$ restricts to a strong monoidal equivalence. Conversely, if $D$ restricts to an equivalence on a monoidal subcategory $\mathcal{C}$, then we must have $\mathcal{C} \subseteq \Refl(\mathcal{M})$. 
Since the equivalence is strong monoidal, we have that $\mu_{x,Dx}$ is an isomorphism for all $x \in \mathcal{C}$. Then by (5), $x \in \Dual(\mathcal{M})$. 
\end{proof}

\begin{remark}
Dualisability is a notion of smallness and coincides with other notions of smallness in many cases. The dualisable objects correspond to the finite-dimensional representations over a finite group; finitely generated projectives in the module category of commutative ring; perfect complexes in the derived category of a commutative ring; and finite spectra in the stable homotopy category.
\end{remark}

\begin{ex}
In general, there are more reflexive objects than dualisable objects. In $\Mod \mathbb{Z}$, the Specker group $\prod_{\mathbb{N}} \mathbb{Z}$ is reflexive but not dualisable. In $\mathcal{D}(R)$, for a commutative ring $R$, the object $\bigoplus_{\mathbb{Z}} \Sigma^i R$ is reflexive but not dualisable.

\end{ex}

\begin{remark}
The reflexive objects do not form a monoidal subcategory in general. Consider $M \coloneq \bigoplus_{i\in \mathbb{Z}} \Sigma^i k \in \mathcal{D}(k)$. This is reflexive but $M \otimes_k M$ is not for cardinality reasons.  
\end{remark}

\begin{prop} \label{reflretract}
Reflexive objects and dualisable objects are closed under retracts.
\end{prop}

\begin{proof}
Suppose $x \xrightarrow{f} y \xrightarrow{g} x$ are such that $gf = 1_x$ and $y$ is reflexive. Then, by naturality of $\eval$, $\eval_x$ is a retract of $\eval_y$. Then the result holds since isomorphisms are closed under retracts. The proof is similar for dualisable objects.
\end{proof}

\begin{remark} \label{enricheddual}
We note that the duality functor $D\colon \mathcal{M} \to \mathcal{M}^{op}$ lifts to an $\mathcal{M}$-enriched functor (see the bottom of page 15 in \cite{Ke82}). This means that for every pair of objects  $x,y \in \mathcal{M}$ there is a morphism in $\mathcal{M}$ 
\[
D_{x,y}\colon \hom(x,y) \to \hom(Dy,Dx)
\]
which is functorial and such that the following diagram commutes.
\[
\begin{tikzcd}[column sep = huge]
\oHom_{\mathcal{M}}(1,\hom(x,y))  \arrow[r,"{\oHom(1,D_{x,y})}"] \arrow[d,"\sim" {anchor=south, rotate=90},swap] & \oHom_{\mathcal{M}}(1, \hom(Dy,Dx)) \arrow[d,"\sim" {anchor=south, rotate=90}] \\
\oHom_{\mathcal{M}}(x,y) \arrow[r,"D"] & \oHom_{\mathcal{M}}(Dy,Dx) 
\end{tikzcd}
\]
where the bottom map is the action of the functor $D\colon \mathcal{M} \to \mathcal{M}^{op}$. Furthermore, Equation \ref{internaladjunction} shows that $D \dashv D^{op}$ is an $\mathcal{M}$-enriched adjunction (see Section 1.11 in \cite{Ke82}) and so the following diagram commutes
\begin{equation} \label{evaladjointdiagram}
\begin{tikzcd}
\hom(x,y) \arrow[rr,"{D_{x,y}}"] \arrow[dr,"{\hom(x,\eval_y)}",swap] & & \hom(Dy,Dx) \\
& \hom(x,DDy) \arrow[ur,"\sim",sloped] & 
\end{tikzcd}
\end{equation}
where the isomorphism is Equation \ref{internaladjunction}.

\end{remark}

\begin{prop} \label{reflequiv}
If $\mathcal{M}$ is a closed symmetric monoidal category and $y \in \mathcal{M}$ is reflexive, then
\[
D_{x,y}\colon \hom(x,y) \to \hom(Dy,Dx)
\]
is an isomorphism for all $x \in \mathcal{M}$.
\end{prop}

\begin{proof}
If $y$ is reflexive, then $\eval_y$ is an isomorphism and so $\hom(x,\eval_y)$ is too. So by Diagram \ref{evaladjointdiagram}, $D_{x,y}$ is an isomorphism.
\end{proof}

\subsection{DG-categories} \label{DGprelims}

Throughout let $k$ be a field and $\mathcal{C}(k)$ will denote the category of chain complexes over $k$. By a DG-category we mean a category enriched in $\mathcal{C}(k)$ i.e.\ instead of a set of morphisms between objects there is a chain complex of morphisms. See \cite{Ke82} for background on enriched category theory and \cite{Kel06} for background on DG-categories. By DG-functor, DG-natural transformation, DG-adjunction we mean the $\mathcal{C}(k)$-enriched notions. For any $M,N \in \mathcal{C}(k)$ there is a complex of morphisms between them denoted $\Hom_k(M,N)$. This makes the category of chain complexes a DG-category which we denote $\underline{\mathcal{C}}(k)$.
\\

A (left) module over a DG-category $\mathcal{A}$ is a DG-functor $\mathcal{A} \to \underline{\mathcal{C}}(k)$. These form a category $\mathcal{C}(\mathcal{A})$ whose morphisms are DG-natural transformations. Between any two objects in $\mathcal{C}(\mathcal{A})$ there is a chain complex of morphisms which can be defined as graded natural transformations as in \cite[Section 2.1.1]{AL17}. This forms a DG-category of $\mathcal{A}$-modules which we will denote $\underline{\mathcal{C}}(\mathcal{A})$. 
\\

If $\mathcal{A}$ is a DG-category, its homotopy category $H^0\mathcal{A}$ is the $k$-linear category with the same objects as $\mathcal{A}$ and with morphisms given by $H^0 \mathcal{A}(a,b)$. We note that $H^0 \underline{\mathcal{C}}(\mathcal{A})$ admits a triangulated structure with shifts and cones defined pointwise. For example, if a $k$-algebra $A$ is viewed as a DG-category with one object, then $H^0 \underline{\mathcal{C}}(A)$ is the triangulated category of chain complexes up to chain homotopy. 
\\

An object $M \in \mathcal{C}(\mathcal{A})$ is acyclic if $M(a)$ is acyclic for every $a \in \mathcal{A}$. The derived category $\mathcal{D}(\mathcal{A})$ of $\mathcal{A}$ is the Verdier quotient of the triangulated category $H^0 \underline{\mathcal{C}}(\mathcal{A})$ at acyclic modules. 
\\

An object $P \in \mathcal{C}(\mathcal{A})$ is $K$-projective if for any acylic object $M \in \mathcal{C}(\mathcal{A})$, the complex $\Hom_{\mathcal{A}}(P,M)$ is acyclic. Let $\underline{\mathcal{D}}(\mathcal{A}) \subseteq \underline{\mathcal{C}}(\mathcal{A})$ denote the full DG-subcategory consisting for $K$-projective modules. The restriction of the localisation functor induces an equivalence
\[
H^0 \underline{\mathcal{D}}(\mathcal{A}) \xrightarrow{\sim} \mathcal{D}(\mathcal{A})
\]

\begin{remark} \label{enrichedcof}

By Corollary 13.2.4 in \cite{Rie14} (see also Theorem 4.20 in \cite{Sch18}) and since $k$ is a field, there is a DG-functor $Q^{\mathcal{A}}: \underline{\mathcal{C}}(\mathcal{A}) \to \underline{\mathcal{C}}(\mathcal{A})$ whose image lies in $\underline{\mathcal{D}}(\mathcal{A})$ and a DG-natural transformation $\varepsilon^{\mathcal{A}}: Q^{\mathcal{A}} \to 1_{\underline{\mathcal{C}}(\mathcal{A})}$ which is a pointwise quasi-isomorphism. The DG-functor $Q^{\mathcal{A}}$ is a DG-enhancement of the cofibrant resolution functor sending a DG-module to its $K$-projective resolution.
\end{remark}

\begin{remark}
The category $\mathcal{D}(\mathcal{A})$ admits a natural triangulated structure and we let $\mathcal{D}^{\perf}(\mathcal{A})$ denote the subcategory of compact objects in $\mathcal{D}(\mathcal{A})$. Similarly we let $\uD^{\perf}(\mathcal{A})$ denote the full DG-subcategory of $\uD(\mathcal{A})$ consisting of objects whose image in $\mathcal{D}(\mathcal{A})$ are perfect.  
\end{remark}

\begin{defn}
If $\mathcal{A}$ is a DG-category, a module $M \in \mathcal{C}(\mathcal{A})$ is cohomologically finite if $H^\ast(M(a))$ is finite-dimensional for every $a \in \mathcal{A}$. We let $\mathcal{D}_{\fd}(\mathcal{A}), \mathcal{C}_{\fd}(\mathcal{A})$ denote the full subcategories of $\mathcal{D}(\mathcal{A})$ and $\mathcal{C}(\mathcal{A})$ consisting of cohomologically finite modules. Similarly we let $\underline{\mathcal{D}}_{\fd}(\mathcal{A}), \underline{\mathcal{C}}_{\fd}(\mathcal{A})$ denote the full DG-subcategories of $\underline{\mathcal{D}}(\mathcal{A}), \underline{\mathcal{C}}(\mathcal{A})$ consisting of cohomologically finite modules. 
\end{defn}

\begin{remark}
Recall $k$ can be viewed as a DG-category with a single object and endomorphism complex given by $k$ in degree zero. The category of modules over this DG-category coincides with the category of chain complexes $\mathcal{C}(k)$. We will write $\mathcal{D}^b(k)$ for $\D^{\perf}(k) = \D_{\fd}(k)$ and $\uD^b(k)$ for $\uD^{\perf}(k) = \uD_{\fd}(k)$.
\end{remark}

\begin{remark} \label{adjointDG}

If $M \in \mathcal{C}(\mathcal{A}^{op} \otimes \mathcal{B})$ is a $\mathcal{B}$-$\mathcal{A}$--bimodule then there is a DG-adjunction
\[
M \otimes_{\mathcal{A}} - : \underline{\mathcal{C}}(\mathcal{A}) \to \underline{\mathcal{C}}(\mathcal{B}): \Hom_{\mathcal{B}}(M,-)
\]
See \cite[Section 2.1.5]{AL17}. These derive to an adjoint pair
\[
M \otimes^{\mathbb{L}}_{\mathcal{A}} - : {\mathcal{D}}(\mathcal{A}) \to {\mathcal{D}}(\mathcal{B}): \RHom_{\mathcal{B}}(M,-)
\]

Given a DG-functor $F\colon \mathcal{A} \to \mathcal{B}$, there is an adjoint triple.
\[
\begin{tikzcd}[column sep = large]
\mathcal{D}(\mathcal{A}) \arrow[r," \mathcal{B} \otimes^{\mathbb{L}}_\mathcal{A} - ",shift left, bend left] \arrow[r,"{\RHom_{\mathcal{A}}(\mathcal{B},-)}",bend right,swap] & \arrow[l,"\Res(F)",swap] \mathcal{D}(\mathcal{B})
\end{tikzcd}
\]
where $\Res(F)$ sends $M \in \mathcal{C}(\mathcal{B})$ to the DG-functor $MF \in \mathcal{C}(\mathcal{A})$. Note that $\Res(F)$ preserves quasi-isomorphisms. 
\end{remark}

\subsection{The Categories $\Hmo$ and $\Hqe$} \label{duals} We wish to consider DG-categories up to some weak notions of equivalence: quasi-equivalences and Morita equivalence. 

\begin{defn} \label{moritaequivadef}
A DG-functor $F: \mathcal{A} \to \mathcal{B}$ is quasi-fully faithful if it induces quasi-isomorphisms $\mathcal{A}(a,a') \to \mathcal{B}(F(a), F(a'))$ for all $a,a' \in \mathcal{A}$. We say $F$ is a quasi-equivalence if it is quasi-fully faithful and the induced functor $H^0(F): H^0\mathcal{A} \to H^0\mathcal{B}$ is essentially surjective. We say $F$ is a Morita equivalence if $F$ induces an equivalence $\Res(F): \mathcal{D}(\mathcal{A}) \simeq \mathcal{D}(\mathcal{B})$ as in  Remark \ref{adjointDG}.
\end{defn}

\begin{remark}
As an example of why we consider these notions of equivalence consider two $k$-algebras $A,B$ which are derived equivalent. By \cite[Theorem 6.4]{Ric89}, there is a DG-functor $\uD(A) \to \uD(B)$ which is a quasi-equivalence but usually not a strict DG-equivalence. Furthermore if we view $A$ an $B$ as DG-categories with a single object, then they are Morita equivalent in the sense of Definition \ref{moritaequivadef}. 
\end{remark}

\begin{remark} \label{YonedaisMorita}
Any quasi-equivalence of DG-categories is a Morita equivalence. A key example of a Morita equivalence is the Yoneda embedding $\mathcal{A} \hookrightarrow \uD^{\perf}(\mathcal{A})^{op}: a \mapsto \mathcal{A}(a,-)$. 
\end{remark}

Let $\dgcat$ denote the category of small DG-categories with DG-functors between them. In order to consider DG-categories up to quasi-equivalence or Morita equivalence, we wish to formally invert these classes of morphisms in $\dgcat$. By Remark \ref{YonedaisMorita}, inverting Morita equivalences amounts to identifying a DG-category with (the opposite of) its perfect derived category. This process is controlled using the theory of model categories. See \cite{Ho99} for background on model category theory.

\begin{remark} \label{hqemodel}
In \cite[Theorem 2.1]{Tab04} and \cite[Theorem 0.7]{tab07}, it is shown that $\dgcat$ admits two different model category structures. In the first, the weak equivalences are quasi-equivalences. We let $\Hqe$ denote the associated homotopy category i.e.\ the category of DG-categories with quasi-equivalences formally inverted. In the second, the weak equivalences are the Morita equivalences and we let $\Hmo$ denote the associated homotopy category. So there are functors 
\[
\dgcat \to \Hqe \to \Hmo
\]
where the first inverts all quasi-equivalences and the second inverts all Morita equivalences. 
\end{remark}

We will use the theory of model categories e.g.\ cofibrant and fibrant resolutions to work with $\Hmo$ and $\Hqe$. For more details see \cite{Ho99}. We will use a result of To\"en which describes the morphisms in $\Hqe$ and $\Hmo$ in terms of bimodules.

\begin{defn}
For DG-categories $\mathcal{A}, \mathcal{B}$, an $\mathcal{A}$--$\mathcal{B}$-bimodule $M \in \mathcal{C}(\mathcal{B}^{op} \otimes \mathcal{A})$ is right quasirepresentable if for every $a \in \mathcal{A}$, the right $\mathcal{B}$-module $M(-,a) \in \mathcal{C}(\mathcal{B}^{op})$ is quasi-isomorphic to a representable module $\mathcal{B}(-,b)$ for some $b \in \mathcal{B}$.
\end{defn}

A right quasirepresentable $\mathcal{A}$--$\mathcal{B}$ bimodule is also called a quasifunctor from $\mathcal{A}$ to $\mathcal{B}$.

\begin{theorem}[Corollary 4.8, \cite{toe06}] \label{toendescription}
For $\mathcal{A},\mathcal{B} \in \Hqe$, there is a bijection between $\oHom_{\Hqe}(\mathcal{A},\mathcal{B})$ and isomorphism classes of right quasirepresentable objects in $\mathcal{D}(\mathcal{B}^{op} \otimes \mathcal{A})$. 
\end{theorem}

Given a DG-functor $F\colon \mathcal{A} \to \mathcal{B}$ the functor $\dgcat \to \Hqe$ sends $F$ to the quasifunctor $M_F \in \mathcal{D}(\mathcal{B}^{op} \otimes \mathcal{A})$ given by $M_F(b,a) = \mathcal{B}(b,F(a))$. Indeed, this can be seen from the proof of Theorem 1.1 in \cite{CS15}.
\\

Similarly, there is a description of the morphism sets in $\Hmo$ as bimodules.

\begin{defn}
For DG-categories $\mathcal{A}, \mathcal{B}$, an $\mathcal{A}$--$\mathcal{B}$ bimodule $M \in \mathcal{C}(\mathcal{B}^{op} \otimes \mathcal{A})$ is right perfect if for every $a \in \mathcal{A}$, $M(-,a) \in \mathcal{C}(\mathcal{B}^{op})$ is a perfect $\mathcal{B}^{op}$-module.
\end{defn}

The following result follows from Theorem \ref{toendescription}. See Proposition 4.4.8 of \cite{To11}.

\begin{theorem} \label{rightperfectdesc}
For $\mathcal{A},\mathcal{B} \in \Hmo$, there is a bijection between $\oHom_{\Hmo}(\mathcal{A},\mathcal{B})$ and isomorphism classes of right perfect modules in $\mathcal{D}(\mathcal{B}^{op} \otimes \mathcal{A})$. 
\end{theorem}

We will be interested in the monoidal structure on $\Hmo$ and $\Hqe$.  

\begin{defn}
The tensor product $\mathcal{A} \otimes \mathcal{B}$ of DG-categories $\mathcal{A}, \mathcal{B}$ has objects given by pairs of objects in $\mathcal{A}$ and $\mathcal{B}$ and morphisms $\mathcal{A}(a,a') \otimes_k \mathcal{B}(b,b')$. 
\end{defn}

This makes $\dgcat$ a closed symmetric monoidal category with unit given by $k$ i.e.\ the DG-category with a single object and endomorphism complex $k$. The internal hom $\hom_{\dgcat}(-,-)$ is the DG-category of DG-functors. The symmetric monoidal and model structures are not compatible but as we are over a field, the tensor product preserves quasi-equivalences and so descends to 
\[
- \otimes - \colon \Hqe \times \Hqe \to \Hqe. 
\]
Over a commutative ring, the tensor product can also be derived to $\Hqe$. The next result follows from a more general statement about the simplicial enrichment associated to the model category. A direct proof is given in \cite{CS15}.

\begin{theorem}[Theorem 1.3, \cite{toe06}] \label{toeinternal} $(\Hqe, \otimes, k)$ is a closed symmetric monoidal category with internal hom $\hom_{\Hqe}(\mathcal{A},\mathcal{B})$ given by the DG-subcategory of $\underline{\mathcal{D}}(\mathcal{B}^{op} \otimes \mathcal{A})$ consisting of right quasirepresentables.  
\end{theorem}

A similar version of the following result is mentioned at the beginning of Section 7 in \cite{toe06}. 

\begin{prop} \label{Toenmodules}
For a DG-category $\mathcal{A}$, there is an isomorphism in $\Hqe$.
\[
\hom_{\Hqe}(\mathcal{A},\uD^b(k)) \simeq \uD_{\fd}(\mathcal{A}). 
\]
\end{prop}

\begin{proof}
Let $Q$ denote the cofibrant replacement functor on $\dgcat$ for the model structure in Remark \ref{hqemodel} whose homotopy category is $\Hqe$. By definition, the functor $\oHom_{\Hqe}(-\otimes \mathcal{A},\uD_{\fd}(k))$ is represented by $ \hom_{\Hqe}(\mathcal{A},\uD_{\fd}(k))$. We will show that $\uD_{\fd}(\mathcal{A})$ also represents this functor and deduce the equivalence. Let $\iso$ denote the isomorphism classes of objects in a category and $\Ho$ the homotopy category of a model category. Consider the following chain of isomorphisms.
\begin{align*}
\oHom_{\Hqe}(-,\uD_{\fd}(\mathcal{A})) & \simeq  \oHom_{\Hqe}(Q(-),\uD_{\fd}(\mathcal{A})) \\
& \simeq \iso \Ho(\hom_{\dgcat}(Q(-),\uC_{\fd}(\mathcal{A}))) \\
& \simeq \iso \Ho(\hom_{\dgcat}(Q(-), \hom_{\dgcat}(\mathcal{A}, \uC_{\fd}(k))))\\
& \simeq \iso \Ho(\hom_{\dgcat}(Q(-) \otimes \mathcal{A}, \uC_{\fd}(k))) \\
& \simeq \iso \Ho(\hom_{\dgcat}(Q(Q(-) \otimes \mathcal{A}), \uC_{\fd}(k)))\\
& \simeq \oHom_{\Hqe}(Q(Q(-) \otimes \mathcal{A}),\uD_{\fd}(k)) \\
& \simeq \oHom_{\Hqe}(- \otimes \mathcal{A},\uD_{\fd}(k))
\end{align*}
The first follows since there is a natural quasi-equivalence $Q(-) \to 1_{\Hqe}$. The second is Lemma 6.2 of \cite{toe06} applied to $M_0 = \uC_{\fd}(\mathcal{A})$. The third is the definition of $\uC_{\fd}(\mathcal{A})$ and the fourth is the underived tensor-hom adjunction. The fifth follows since $\Ho(\hom_{\dgcat}(\mathcal{B},\uC_{\fd}(k)) = \uD_{\fd}(\mathcal{B})$ and the quasi-equivalence $Q(\mathcal{B}) \to \mathcal{B}$  induces a quasi-equivalence $\uD_{\fd}(\mathcal{B}) \simeq \uD_{\fd}(Q(\mathcal{B}))$ for any DG-category $\mathcal{B}$. The sixth is another application of Lemma 6.2 of loc.\ cit.\ 
\end{proof}

The symmetric monoidal structure can be derived to $\Hmo$ by setting $\mathcal{A} \otimes \mathcal{B} \coloneq \uD^{\perf}(\mathcal{A} \otimes \mathcal{B})$. For more details see \cite[Section 4.4]{To11}. The following result follows from Theorem \ref{toeinternal} see \cite[Exercise 33]{To11}.

\begin{theorem}  \label{toeinternalhmo}
The symmetric monoidal category $\Hmo$ is closed with internal hom given by $\hom_{\Hmo}(\mathcal{A},\mathcal{B}) \simeq \hom_{\Hqe}(\mathcal{A},\mathcal{D}^{\perf}(\mathcal{B}))$ which is quasi-equivalent to the DG-subcategory of $\uD(\mathcal{B}^{op} \otimes \mathcal{A})$ consisting of right-perfect modules.
\end{theorem}

\begin{remark} 
The unit in the closed symmetric monoidal category $\Hmo$ is $k$ but as Morita equivalences are isomorphisms in $\Hmo$ we have that $k \simeq \uD^b(k)$. Therefore the duality functor on $\Hmo$ applied to a DG-category $\mathcal{A}$ is 
\[
\hom_{\Hmo}(\mathcal{A},k) \simeq \hom_{\Hmo}(\mathcal{A},\uD^b(k)) \simeq \hom_{\Hqe}(\mathcal{A},\uD^b(k))) \simeq \uD_{\fd}(\mathcal{A})
\]
using Theorem \ref{toeinternalhmo} and Proposition \ref{Toenmodules}. 
\end{remark}

\begin{remark}
Note that Proposition \ref{Toenmodules} could also be extracted directly from Theorem \ref{toeinternalhmo}. However we include the direct proof of Proposition \ref{Toenmodules} as it will be used in the proof of Lemma \ref{isofunctors}.
\end{remark}

\begin{remark}\label{isoclassesobj}
By Theorems \ref{toeinternalhmo} and \ref{rightperfectdesc}, there is a bijection between $\oHom_{\Hmo}(\mathcal{A},\mathcal{B})$ and isomorphism classes of objects in $H^0\hom_{\Hqe}(\mathcal{A},\mathcal{B})$.
\end{remark}

\subsection{Hochschild Cohomology}
\label{HHsectionfd}

Hochschild cohomology is a fundamental invariant appearing in algebra, geometry and topology. We take the following as our definition of the Hochschild cohomology of a DG-category.

\begin{defn}
The Hochschild cohomology of a DG-category $\mathcal{A}$ is the cohomology of the complex
\[
HH(\mathcal{A}) \coloneq \RHom_{\mathcal{A}^e}(\mathcal{A},\mathcal{A})
\]
where $\mathcal{A}$ is viewed as the diagonal bimodule.
\end{defn}

We will write $HH^\ast(\mathcal{A})$ for $H^\ast (HH(\mathcal{A}))$. Note that $HH(\mathcal{A})$ is an endomorphism object in a DG-category and so is a DG-algebra.

\begin{ex}
Viewing a $k$-algebra as a DG-category with one object, this clearly generalises the usual notion of Hochschild cohomology of an algebra. The Hochschild cohomology of a quasicompact separated scheme $X$ is defined as $\RHom_{X \times X}(\Delta, \Delta)$ where $\Delta$ denotes the pushforward of the structure sheaf along the diagonal map. The HKR Theorem of \cite{Ma01},\cite{Cal05} gives a geometric description of the Hochschild cohomology of $X$. It follows from \cite[Theorem 8.9]{toe06} that the Hochschild cohomology of $X$ coincides with the Hochschild cohomology of the DG-category $\uD^{\perf}(X)$.
\end{ex}

\begin{remark} There is also a Hochschild chain complex for DG-categories. See \cite{LVdB05}.
\end{remark}

We recall a result of To\"en showing that Hochschild cohomology can be encoded in $\Hqe$ and $\Hmo$ (as defined in Section \ref{duals}). 

\begin{theorem}[Corollary 8.1, \cite{toe06}] \label{toehh}
If $\mathcal{A}$ is a DG-category, then
\[
HH(\mathcal{A}) \simeq \hom_{\Hqe}(\mathcal{A},\mathcal{A})(1_\mathcal{A},1_{\mathcal{A}}) \simeq \hom_{\Hmo}(\mathcal{A},\mathcal{A})(1_\mathcal{A},1_{\mathcal{A}})
\]
the endomorphism DG-algebras of $1_\mathcal{A} \in \hom_{\Hqe}(\mathcal{A},\mathcal{A})$ and in $\hom_{\Hmo}(\mathcal{A},\mathcal{A})$
\end{theorem}

\begin{remark} \label{Hochsperf}
Since $HH(\mathcal{A}) \simeq HH(\mathcal{A}^{op})$, it follows from Remark \ref{YonedaisMorita} and Theorem \ref{toehh}, that $HH(\mathcal{A}) \simeq HH(\uD^{\perf}(\mathcal{A}))$.
\end{remark}

\section{Reflexive DG-categories via an Explicit Evaluation Map}
\label{KSreflsect}

Reflexive DG-categories were introduced in \cite{KS22} as those DG-categories $\mathcal{A}$ for which a certain evaluation map $\mathcal{A} \to \uD_{\fd}\uD_{\fd}(\mathcal{A})$ is a quasi-equivalence. We give a detailed construction of the functor $\uD_{\fd}(-): \Hmo \to \Hmo^{op}$ (where $\Hmo$ is as in Section \ref{duals}) and the evaluation natural transformation $\ev: 1_{\Hmo} \to \uD_{\fd}\uD_{\fd}(-)$ between functors $\Hmo \to \Hmo$ which is skipped in \cite{KS22}.

\begin{remark}
For a DG-category $\mathcal{A}$, let $i^\mathcal{A}$ denote the DG-functor given by the inclusion $\uD(\mathcal{A}) \subseteq \uC(\mathcal{A})$. By Remark \ref{enrichedcof}, there are DG-functors
\[
\begin{tikzcd}
\uD(\mathcal{A}) \arrow[r,hook,"i^{\mathcal{A}}"] & \uC(\mathcal{A})  \arrow[r,"{Q^{\mathcal{A}}}"] & \uD(\mathcal{A})
\end{tikzcd}
\]
and a DG-natural transformation $\varepsilon^{\mathcal{A}}\colon i^{\mathcal{A}} Q^{\mathcal{A}} \to 1_{\uC(\mathcal{A})}$ which is an objectwise quasi-isomorphism. Note these DG-functors restrict to cohomologically finite modules. We will use them to write down an explicit evaluation map. 
\end{remark}

To construct $\uD_{\fd}$ as a functor on $\Hmo$ and to write down an explicit evaluation map we will use the following technical lemma. Recall that $\Hqe$ is the precursor to $\Hmo$ where quasi-equivalences are inverted.

\begin{lemma} \label{naturalitylemma}
Suppose $F,G\colon \mathcal{A} \to \mathcal{B}$ are DG-functors and $\alpha\colon F \to G$ is a DG-natural transformation such that for every $a \in \mathcal{A}$, the map $\alpha_a \in \oHom_{H^0\mathcal{B}}(F(a),G(a))$ is an isomorphism in the category $H^0\mathcal{B}$. Then $F$ and $G$ represent the same morphism in $\Hqe$.
\end{lemma}

\begin{proof}
By Theorem \ref{toendescription}, the morphisms in $\Hqe$ corresponding to $F$ and $G$ are represented by quasifunctors are $M_F,
M_G \in \mathcal{C}(\mathcal{B}^{op} \otimes \mathcal{A})$ where $M_F(b,a) = \mathcal{B}(b,F(a))$ and $M_G(b,a) = \mathcal{B}(b,G(a))$. There are composition maps for any $b \in \mathcal{B}$ $a \in \mathcal{A}$
\[
 \mathcal{B}(b,F(a)) \otimes_k \mathcal{B}(F(a),G(a)) \to \mathcal{B}(b,G(a)).
\]
Composing with $\alpha_a \in \mathcal{B}(F(a),G(a))$ produces chain maps $M_F(b,a) \to M_G(b,a)$ which are natural and give rise to map $M_F \to M_G \in \mathcal{C}(\mathcal{B}^{op} \otimes \mathcal{A})$. Taking cohomology produces maps for every $i$
\[
 H^i\mathcal{B}(b,F(a)) \otimes_k H^0\mathcal{B}(F(a),G(a)) \to H^i\mathcal{B}(b,G(a)).
\]
Since $\alpha_a$ is an isomorphism in $H^0 \mathcal{B}$, composing with the class $\alpha_a$ represents induces an isomorphism $H^\ast M_F(b,a) \to H^\ast M_G(b,a)$. Therefore, $M_F \simeq M_G \in \mathcal{D}(\mathcal{B}^{op} \otimes \mathcal{A})$. Hence, the isomorphism class of $M_F$ equals the isomorphism class of $M_G$ and the corresponding elements of $\oHom_{\Hqe}(\mathcal{A},\mathcal{B})$ are equal.
\end{proof}

\begin{defn} \label{Dfdfunct}
Define a functor $\uD_{\fd}(-)\colon \dgcat \to \Hqe^{op}$ which sends $\mathcal{A} \mapsto \uD_{\fd}(\mathcal{A})$ and $F\colon \mathcal{A} \to \mathcal{B}$ to the morphism represented by the DG-functor
\[
\begin{tikzcd}
\uD_{\fd}(\mathcal{B}) \arrow[r,hook,"i^{\mathcal{B}}"] & \uC_{\fd}(\mathcal{B}) \arrow[r,"{\Res(F)}"] &\uC_{\fd}(\mathcal{A}) \arrow[r,"{Q^{\mathcal{A}}}"] & \uD_{\fd}(\mathcal{A})
\end{tikzcd}
\]
where $\Res(F)(M) = MF$. 
\end{defn}

\begin{prop}
The functor $\uD_{\fd}(-)\colon \dgcat \to \Hqe^{op}$ in Definition \ref{Dfdfunct} is well defined.
\end{prop}

\begin{proof}
We check $\uD_{\fd}(-)$ is functorial. For a DG-category $\mathcal{A}$, $1_{\mathcal{A}}$ is sent to the quasifunctor representing the DG-functor $Q^{\mathcal{A}}i^\mathcal{A}\colon \uD_{\fd}(\mathcal{A}) \to \uD_{\fd}(\mathcal{A})$. There is a DG-natural transformation $\varepsilon^{\mathcal{A}}\colon Q^{\mathcal{A}}i^\mathcal{A} \to 1_{\uD_{\fd}(\mathcal{A})}$ such that for any $M \in \uD_{\fd}(\mathcal{A})$ the map $Q^{\mathcal{A}}(M) \to M$ is a quasi-isomorphism. Therefore, the morphisms $\varepsilon^{\mathcal{A}}_M \in \oHom_{H^0\uD_{\fd}(\mathcal{A})}(Q^{\mathcal{A}}(M),M)$ are isomorphisms so by Lemma \ref{naturalitylemma}, $\uD_{\fd}(1_{\mathcal{A}}) = 1_{\uD_{\fd}(\mathcal{A})}$ as elements of $\oHom_{\Hqe}(\uD_{\fd}(\mathcal{A}),\uD_{\fd}(\mathcal{A}))$. Given DG-functors $F\colon \mathcal{A} \to \mathcal{B}$ and $G\colon \mathcal{B} \to \mathcal{C}$, there is a DG-natural transformation 
\[
\alpha\coloneq Q^\mathcal{A}\Res(F) \varepsilon^\mathcal{B}_{\Res(G)i^\mathcal{C}} \colon \uD_{\fd}(F) \uD_{\fd}(G)  \to \uD_{\fd}(GF).
\]
For any $M \in \uD_{\fd}(\mathcal{C})$, $\varepsilon^{\mathcal{B}}_{{\Res(G)}i^{\mathcal{C}}M}$ is a quasi-isomorphism in $\uC_{\fd}(\mathcal{B})$. Restricting along $F$ and applying $Q^{\mathcal{A}}$ preserves quasi-isomorphisms and so 
\[
\alpha_M \in \oHom_{H^0\uD_{\fd}(\mathcal{A})}(Q^{\mathcal{A}}(Q^{\mathcal{B}}(MG)F), Q^{\mathcal{A}}(MGF))
\]
is an isomorphism for every $M \in \mathcal{D}_{\fd}(\mathcal{C})$. Again by Lemma \ref{naturalitylemma} we are done. 
\end{proof}

\begin{prop} \label{Dfdfunctor}
$\uD_{\fd}(-)$ descends to functors $\Hqe \to \Hqe^{op}$ and $\Hmo \to \Hmo^{op}$.
\end{prop}

\begin{proof}
If $\uD_{\fd}$ descends to $\Hmo$, then it certainly descends to $\Hqe$ by Remark \ref{YonedaisMorita}. We must show that if $F\colon \mathcal{A} \to \mathcal{B}$ is a Morita equivalence, then $\uD_{\fd}(F)$ is an isomorphism in $\Hmo$. This is equivalent to showing that the DG-functor $\uD_{\fd}(F)\colon \uD_{\fd}(\mathcal{B}) \to \uD_{\fd}(\mathcal{A})$ is a quasi-equivalence. If $F$ is a Morita equivalence, then the unbounded restriction functor $\D(\mathcal{B}) \to \D(\mathcal{A})$ is an equivalence by definition. This equivalence preserves the perfects as they are the compact objects. Note that $M \in \mathcal{D}_{\fd}(\mathcal{A})$ if and only if $\RHom_{\mathcal{A}}(P,M) \in \mathcal{D}^b(k)$ for every $P \in \mathcal{D}^{\perf}(\mathcal{A})$. Therefore, the equivalence restricts to an equivalence $\mathcal{D}_{\fd}(\mathcal{B}) \simeq \mathcal{D}_{\fd}(\mathcal{A})$. 
\end{proof}

\begin{prop} \label{evdefn}
There are natural transformations $\ev\colon 1_{\Hqe} \to \uD_{\fd}\uD_{\fd}(-)$ and $\ev\colon 1_{\Hmo} \to \uD_{\fd}\uD_{\fd}(-)$ such that for each DG-category $\mathcal{A}$ and $a \in \mathcal{A}$, $M \in \uD_{\fd}(\mathcal{A})$, $\ev_{\mathcal{A},a}(M) \simeq M(a)$.

\end{prop}

\begin{proof}
For $\mathcal{A} \in \dgcat$, define
\[
\overline{\ev}_\mathcal{A}\colon \mathcal{A} \to \uC_{\fd}\uC_{\fd}(\mathcal{A})
\]
as follows. For $a \in \mathcal{A}$ set 
\[
\overline{\ev}_{\mathcal{A},a}\colon \uC_{\fd}(\mathcal{A})\to \uC_{\fd}(k); M \mapsto M(a);\ \beta \mapsto \beta_a.
\]
This defines a DG-functor and so $\overline{\ev}_{\mathcal{A},a} \in \uC_{\fd}\uC_{\fd}(\mathcal{A})$. Given some $f \in \mathcal{A}(a,b)^n$, set $\overline{\ev}_{\mathcal{A},f}\colon \overline{\ev}_{\mathcal{A},a} \to \overline{\ev}_{\mathcal{A},b}$ as the graded natural transformation whose value at $M \in \uC_{\fd}(\mathcal{A})$ is
\[
\overline{\ev}_{\mathcal{A},a}(M) = M(a) \xrightarrow{M(f)} M(b) = \overline{\ev}_{\mathcal{A},b}(M).
\]
Then $f \mapsto \overline{\ev}_{\mathcal{A},f}$ is a chain map, and this constructs a DG-functor $\overline{\ev}_{\mathcal{A}}$. Define $\ev_{\mathcal{A}}$ as the composition
\[
\ev_{\mathcal{A}}\colon \mathcal{A} \xrightarrow{\overline{\ev}_{\mathcal{A}}} \uC_{\fd}\uC_{\fd}(\mathcal{A}) \xrightarrow{\Res(i^\mathcal{A})} \uC_{\fd}\uD_{\fd}(\mathcal{A}) \xrightarrow{Q^{\uD_{\fd}(\mathcal{A})}} \uD_{\fd}\uD_{\fd}(\mathcal{A})
\]
and view it as a morphism in $\Hqe$ and $\Hmo$ by taking its image along the functors $\dgcat \to \Hqe \to \Hmo$. We now check naturality. First, note that if $\alpha\colon F \to G$ is a DG-natural transformation between DG-functors $F,G\colon \mathcal{A} \to \mathcal{B}$, there is a DG-natural transformation $\Res(\alpha)\colon \Res(F) \to \Res(G)$ defined by 
\[
\Res(\alpha)_M(a) \colon \Res(F)(M)(a) = MF(a) \xrightarrow{M(\alpha_a)} MG(a) = \Res(G)(M)(a) 
\]
for $M \in \uC(\mathcal{B})$ and $a \in \mathcal{A}$. Now, given $F\colon \mathcal{A} \to \mathcal{B} \in \Hqe$, note that since every object of $\Hqe$ is isomorphic to a cofibrant object, we can assume that $\mathcal{A}$ is cofibrant and so $F$ is a DG-functor. We will define a DG-natural transformation $\phi$ of the following form.
\[
\begin{tikzcd}[column sep = large]
\mathcal{A} \arrow[d,"\ev_{\mathcal{A}}"] \arrow[r,"F"] & \arrow[d,"\ev_{\mathcal{B}}"] \mathcal{B} \\ 
\uD_{\fd}\uD_{\fd}\mathcal{A} \arrow[r,"\uD_{\fd}\uD_{\fd}(F)",swap] \arrow[ru,Rightarrow,"\phi" description, shorten >= 2ex] &\uD_{\fd}\uD_{\fd}\mathcal{B} 
\end{tikzcd}
\]
First, let $\gamma^{\mathcal{A}} = \gamma\colon \Res(Q^\mathcal{A}) i^{\uD_{\fd}(\mathcal{A})} Q^{\uD_{\fd}(\mathcal{A})} \Res(i^\mathcal{A}) \to 1_{\uC_{\fd}\uC_{\fd}(\mathcal{A})}$ denote the DG-natural transformation defined as 
\[
\begin{tikzcd}[column sep = huge]
\uC_{\fd}\uC_{\fd}(\mathcal{A}) \arrow[r,"\Res(i^\mathcal{A})"] & \uC_{\fd}\uD_{\fd}(\mathcal{A}) \arrow[r,"i^{\uD_{\fd}(\mathcal{A})}Q^{\uD_{\fd}(\mathcal{A})}"{yshift = 5pt}] \arrow[dr,Rightarrow,"{\varepsilon^{\uD_{\fd}(\mathcal{A})}}" description] & \uC_{\fd}\uD_{\fd}(\mathcal{A}) \arrow[r,"{\Res(Q^\mathcal{A})}"] & \uC_{\fd}\uC_{\fd}(\mathcal{A}) \\
\uC_{\fd}\uC_{\fd}(\mathcal{A}) \arrow[r,"\Res(i^\mathcal{A})"] \arrow[rrrd,Rightarrow,"{\Res(\varepsilon^\mathcal{A})}" description,shorten <= 4ex,shorten >= 7ex] \arrow[u,equal] & \uC_{\fd}\uD_{\fd}(\mathcal{A}) \arrow[r,equal] \arrow[u,equal]  & \uC_{\fd}\uD_{\fd}(\mathcal{A}) \arrow[r,"\Res(Q^\mathcal{A})"] \arrow[u,equal] & \uC_{\fd}\uC_{\fd}(\mathcal{A}) \arrow[u,equal] \\
\uC_{\fd}\uC_{\fd}(\mathcal{A}) \arrow[u,equal]\arrow[rrr,equal] & & & \uC_{\fd}\uC_{\fd}(\mathcal{A}) \arrow[u,equal]
\end{tikzcd} \label{doubletrouble} 
\]
Then let $\phi$ be the DG-natural transformation
\[
\begin{tikzcd}[column sep = huge]
\mathcal{A} \arrow[r,equal] \arrow[d,"\overline{\ev}_{\mathcal{A}}"] & \mathcal{A} \arrow[d,"\overline{\ev}_{\mathcal{A}}"] \arrow[r,"F",swap] &  \mathcal{B}  \arrow[dd,"\overline{\ev}_{\mathcal{B}}"]  \arrow[ddr,"\ev_{\mathcal{B}}"] & 
\\
\uC_{\fd}\uC_{\fd}(\mathcal{A}) \arrow[r,equal] \arrow[d,equal] & \uC_{\fd}\uC_{\fd}(\mathcal{A}) \arrow[d,equal] & 
\\
\uC_{\fd}\uC_{\fd}(\mathcal{A}) \arrow[r] \arrow[ur,Rightarrow,"\gamma" description,shorten >= 1ex]  & \uC_{\fd}\uC_{\fd}(\mathcal{A}) \arrow[r,"\Res(\Res(F))",swap] & \uC_{\fd}\uC_{\fd}\mathcal{B}  \arrow[r,"Q^{\uD_{\fd}(\mathcal{B})} \Res(i^{\mathcal{B}})",swap] & \uD_{\fd}\uD_{\fd}(\mathcal{B})
\end{tikzcd}
\]
where the middle square commutes by naturality of $\overline{\ev}$ and the right triangle commutes by definition. By Lemma \ref{naturalitylemma}, it is enough to show that $\phi_a$ is a quasi-isomorphism. For any $a \in \mathcal{A}$, we claim that $\Res\Res(F) \gamma_a$ is a quasi-isomorphism of $\uC_{\fd}(\mathcal{B})$ modules. If $M \in \uC_{\fd}(\mathcal{B})$, then $(\Res\Res(F) \gamma_a)_M$ is the composition

\[
\begin{tikzcd}
\left(\Res\Res(F) \Res(Q^\mathcal{A}) i^{\uD_{\fd}(\mathcal{A})}Q^{\uD_{\fd}(\mathcal{A})} \Res(i^\mathcal{A}) \overline{\ev}_{\mathcal{A},a}\right)(M) \arrow[d,equal] \\
i^{\uD_{\fd}(\mathcal{A})} Q^{\uD_{\fd}(\mathcal{A})}(\overline{\ev}_{\mathcal{A},a} i^\mathcal{A})(Q^\mathcal{A}(MF)) \arrow[d,"{\varepsilon^{\uD_{\fd}(\mathcal{A})}_{\overline{\ev}_{\mathcal{A},a}i^{\mathcal{A}}}(Q^{\mathcal{A}}(MF))}"] \\
(\overline{\ev}_{\mathcal{A},a}i^\mathcal{A})(Q^\mathcal{A}(MF)) = i^\mathcal{A} Q^\mathcal{A}(MF)(a) \arrow[d,"{\varepsilon^{\mathcal{A}}_{MF}(a)}"] \\
MF(a)
\end{tikzcd}
\]
This is a quasi-isomorphism since   $\varepsilon^{\uD_{\fd}(\mathcal{A})}_{\overline{\ev}_{\mathcal{A},a}i^{\mathcal{A}}}$ is a quasi-isomorphism of $\uD_{\fd}(\mathcal{A})$-modules and $\varepsilon^{\mathcal{A}}_{MF}$
is a quasi-isomorphism of $\mathcal{A}$-modules. Hence, $\Res\Res (F) \gamma_a$ is a quasi-isomorphism for all $a$. Therefore, $\Res(i^\mathcal{B}) \Res \Res(F) \gamma_a$ is a quasi-isomorphism for all $a \in \mathcal{A}$ as restriction preserves quasi-isomorphisms. Therefore, 
\[
\phi_a = Q^{\uD_{\fd}(\mathcal{B})} \Res(i^\mathcal{B}) \Res \Res(F) \gamma_a 
\]
is a quasi-isomorphism since $Q^{\uD_{\fd}(\mathcal{B})}$ preserves quasi-isomorphisms. As they commute in $\Hqe$, the naturality diagrams also commute in $\Hmo$. 
\end{proof}

The following is a version of \cite[Definition 3.11]{KS22} extended to all DG-categories instead of just pretriangulated DG-categories. 

\begin{defn} \label{reflsemirefldefn}
A DG-category $\mathcal{A}$ is reflexive if the DG-functor
\[
\ev_{\mathcal{A}}: \mathcal{A} \to \uD_{\fd} \uD_{\fd}(\mathcal{A})
\]
is a Morita equivalence. We say that $\mathcal{A}$ is semireflexive if $\ev_{\mathcal{A}}$ is quasi-fully faithful.
\end{defn}

\begin{prop}\label{Moritainvrefl}
If $\mathcal{A}$ is Morita equivalent to $\mathcal{B}$, then $\mathcal{A}$ is reflexive if and only if $\mathcal{B}$ is reflexive.
\end{prop}

\begin{proof}
Note that $\ev_{\mathcal{A}}$ is a Morita equivalence if and only if it is an isomorphism in $\Hmo$. By assumption $\mathcal{A}$ is isomorphic to $\mathcal{B}$ in $\Hmo$. Since $\ev$ is a natural transformation between endofunctors on $\Hmo$, $\ev_{\mathcal{A}}$ is an isomorphism if and only if $\ev_{\mathcal{B}}$ is.
\end{proof}

\begin{remark}
There are dualities $\mathcal{D}^{\perf}(\mathcal{A}^{op}) \simeq \mathcal{D}^{\perf}(\mathcal{A})^{op}$ and $\mathcal{D}_{\fd}(\mathcal{A})^{op} \simeq \mathcal{D}_{\fd}(\mathcal{A}^{op})$ which can be composed with $\ev_\mathcal{A}$ to give equivalent formulations of reflexivity. These are studied in Lemma 3.10 of \cite{KS22} and in particular, it is shown that $\mathcal{A}$ is reflexive if and only if $\mathcal{A}^{op}$ is. Recall by Remark \ref{YonedaisMorita}, that $\mathcal{A}$ is Morita equivalent to $\uD^{\perf}(\mathcal{A})^{op}$ and so by Proposition \ref{Moritainvrefl}, $\mathcal{A}$ is reflexive if and only if $\uD^{\perf}(\mathcal{A})$ is reflexive.  
\end{remark}

\begin{prop}
If $\mathcal{A}$ is a DG-category, there is a triangulated functor
\[
\mathcal{D}^{\perf}(\mathcal{A})^{op} \to \D_{\fd}(\uD_{\fd}(\mathcal{A}))
\]
which sends $M \mapsto \RHom_{\mathcal{A}}(M,-)$ and $\mathcal{A}$ is reflexive if and only if it is an equivalence.
\end{prop}

\begin{proof}
By Proposition \ref{Moritainvrefl} and Remark \ref{YonedaisMorita}, $\mathcal{A}$ is reflexive if and only if $\mathcal{D}^{\perf}(\mathcal{A})^{op}$ is. So $\mathcal{A}$ is reflexive if and only if the DG-functor $\ev_{\uD^{\perf}(\mathcal{A})^{op}}$ is a Morita equivalence. Any DG-functor between pretriangulated idempotent complete DG-categories is a Morita equivalence if and only if it is a quasi-equivalence. This occurs if and only if it induces an equivalence of categories on $H^0$. Therefore $\mathcal{A}$ is reflexive if and only if the following triangulated functor is an equivalence
\[
H^0(\ev_{\uD^{\perf}(\mathcal{A})^{op}}): \mathcal{D}^{\perf}(\mathcal{A})^{op} \to \D_{\fd}(\uD_{\fd}(\uD^{\perf}(\mathcal{A})^{op})) 
\]
The Yoneda embedding $\mathcal{A} \hookrightarrow \mathcal{D}^{\perf}(\mathcal{A})^{op}$ is a Morita equivalence and so induces a quasi-equivalence $\uD_{\fd}(\mathcal{A}) \simeq \uD_{\fd}(\uD^{\perf}(\mathcal{A})^{op})$. Therefore $\mathcal{A}$ is reflexive if and only if the composite
\[ 
\mathcal{D}^{\perf}(\mathcal{A})^{op} \to \D_{\fd}(\uD_{\fd}(\uD^{\perf}(\mathcal{A})^{op})) \simeq \D_{\fd}(\uD_{\fd}(\mathcal{A}))
\]
is an equivalence. The equivalence is given by restricting along the map $\uD_{\fd}(\mathcal{A}) \to \uD_{\fd}(\uD^{\perf}(\mathcal{A})^{op}): N \mapsto \RHom_{\mathcal{A}}(-,N)$. It follows that the composite sends $M$ to $\RHom_{\mathcal{A}}(M,-)$ as claimed.
\end{proof}

\begin{remark}
By Lemma 3.13 in \cite{KS22}, if $\mathcal{A}$ is a reflexive DG-category, then so is $\uD_{\fd}(\mathcal{A})$. This will also follow from Proposition \ref{reflexivechar} and Theorem \ref{samedef}.
\end{remark}

\begin{ex}
Many commonly studied DG-categories in algebra, geometry and topology turn out to be reflexive. For example, proper connective DG-algebras and proper schemes (see \cite[Section 6]{KS22} or for proofs over any field: \cite[Section 5]{Goo24c} and \cite[Corollary 3.4.3]{BGO25}). Many other examples are studied and surveyed in \cite{BGO25}.
\end{ex}

\subsection{Proper DG-categories}

In this subsection we generalise Lemma 3.14 of \cite{KS22} which shows that for proper DG-categories, the evaluation functor is related to the Yoneda embedding. 

\begin{remark}
Suppose $\mathcal{A}$ is a proper DG-category and there is a DG-subcategory $\mathcal{A}^{op} \subseteq \mathcal{B} \subseteq \uD_{\fd}(\mathcal{A})$, for example $\mathcal{B} = \mathcal{D}^{\perf}(\mathcal{A})$. The Yoneda embedding of $\mathcal{B}$ restricted to $\mathcal{A}$ factors as
\[
\begin{tikzcd}
\mathcal{B}^{op} \arrow[r,hook,"Y"]  & \uD(\mathcal{B})\\
\mathcal{A} \arrow[r,hook,"Y\mid_{\mathcal{A}}"] \arrow[u,hook] & \uD_{\fd}(\mathcal{B}) \arrow[u,hook]
\end{tikzcd}
\] 
since $\RHom_{\mathcal{A}}(\mathcal{A}(a,-),b) \in \mathcal{D}^b(k)$ for all $a \in \mathcal{A}$ and $b \in \mathcal{B}$. 
\end{remark}

\begin{lemma} \label{yonedalemma}
Suppose $\mathcal{A}$ is a proper DG-category and $
\mathcal{A}^{op} \subseteq \mathcal{B} \subseteq \uD_{\fd}(\mathcal{A})$, and denote the inclusion by $j\colon \mathcal{B} \hookrightarrow \uD_{\fd}(\mathcal{A})$. Then the composite  
\[
\mathcal{A} \xrightarrow{\ev_{\mathcal{A}}} \uD_{\fd}\uD_{\fd}(\mathcal{A}) \xrightarrow{\uD_{\fd}(j)} \uD_{\fd}(\mathcal{B})
\]
and the restricted Yoneda embedding $Y\mid_{\mathcal{A}}\colon \mathcal{A} \to \uD_{\fd}(\mathcal{B})$
represent the same morphism in $\Hqe$.
\end{lemma}

\begin{proof}
Consider the diagram below.
\[
\begin{tikzcd}[column sep = large]
\mathcal{A} \arrow[dr,"Y\mid_{\mathcal{A}}",hook,swap] \arrow[r,"\overline{\ev}_{\mathcal{A}}"] & \uC_{\fd}\uC_{\fd}(\mathcal{A}) \arrow[r,"\Res(i^{\mathcal{A}})"]& \uC_{\fd}\uD_{\fd}(\mathcal{A})  \arrow[r,"Q^{\uD_{\fd}(\mathcal{A})}"]  \arrow[dr,equal,bend right,""{name = e1}]  & \uD_{\fd}\uD_{\fd}(\mathcal{A}) \arrow[d,"i^{\uD_{\fd}(\mathcal{A})}"] \arrow[from = 1-4,to = e1, Rightarrow, "\varepsilon^{\uD_{\fd}(\mathcal{A})}"] \\
&   \uD_{\fd}(\mathcal{B}) \arrow[rrdd,equal,""{name = e2},bend right] \arrow[rrd,hook,"i^{\mathcal{B}}"name={eyeb},swap] &  & \uC_{\fd}\uD_{\fd}(\mathcal{A}) \arrow[d,"\Res(j)"] \\
 & & & \uC_{\fd}(\mathcal{B})  \arrow[d,"Q^{\mathcal{B}}"] \\
 & & & \uD_{\fd}(\mathcal{B}) \arrow[from = e1, to = eyeb,Rightarrow,shorten >= 1ex,shorten <= 1ex] \arrow[from = eyeb,to = e2,Rightarrow,"\varepsilon^{\mathcal{B}}",shorten <= 1ex]
\end{tikzcd}
\] 
There is nothing derived about the pentagon in this diagram. Both DG-functors $\mathcal{A} \to \uC_{\fd}(\mathcal{B})$ can be explicitly described. For example, on objects the top route sends $a \in \mathcal{A}$ to the $\mathcal{B}$-module 
\[
\mathcal{B} \to \mathcal{D}^b(k); b \mapsto b(a); \beta \mapsto \beta_a
\]
The bottom sends $a$ to 
\[
\mathcal{B} \to \mathcal{D}^b(k); b \mapsto \mathcal{B}(\mathcal{A}(a,-),b); \beta \mapsto \mathcal{B}(\mathcal{A}(a,-),b) ; 
\]
By the  $\mathcal{C}(k)$-enriched Yoneda lemma, there is a natural isomorphism between these two DG-functors $\mathcal{A} \to \uC_{\fd}(\mathcal{B})$ as indicated. Since $\varepsilon^{\mathcal{B}}$ is a pointwise quasi-isomorphism, it remains to see that $Q^{\mathcal{B}} \Res(j) \varepsilon^{\uD_{\fd}(\mathcal{A})}$ is too. This holds since $\varepsilon^{\uD_{\fd}(\mathcal{A})}$ is a pointwise quasi-isomorphism which is preserved by $\Res(j)$ and so is sent to a pointwise quasi-isomorphism by $Q^{\mathcal{B}}$. So by Lemma \ref{naturalitylemma} we are done.
\end{proof}

We recover Lemma 3.14 of \cite{KS22} by taking $\mathcal{B} = \uD_{\fd}(\mathcal{A})$:

\begin{prop} \label{propersemiref}
Any proper DG-category is semireflexive.
\end{prop}

\section{Reflexive DG-categories as Reflexive Objects} \label{reflDGsec}

The goal of this section is to show that reflexive DG-categories are the reflexive objects in the closed symmetric monoidal category $\Hmo$. Recall that $\Hmo$ is defined in Section \ref{duals} and $\ev$ is constructed in Proposition \ref{evdefn}.

\begin{prop} \label{adjunctionDfd}
There is an adjunction
\[ \begin{tikzcd}
\uD_{\fd}(-)\colon \Hmo \arrow[r,shift left] & \arrow[l,shift left] \Hmo^{op} \colon \uD_{\fd}(-)
\end{tikzcd} \]
whose unit and counit are both given by the natural transformation $\ev$.
\end{prop}

\begin{proof}
In order to prove that $\uD_{\fd}(-) \dashv \uD_{\fd}(-)$ is an adjunction, we will prove that the corresponding triangle identities holds with  candidate unit and counit given by 
$\ev_{\uD_{\fd}(\mathcal{A})}$. Both triangle identities state that $\uD_{\fd}(\ev_{\mathcal{A}}) \ev_{\uD_{\fd}}(\mathcal{A}) = 1_{\uD_{\fd}(\mathcal{A})}$ in $\Hmo$. To do this, we construct a DG-natural transformation of the form
\[
\begin{tikzcd}[column sep = large, row sep = large]
\uD_{\fd}(\mathcal{A}) \arrow[r,"\ev_{\uD_{\fd}(\mathcal{A})}"] \arrow[dr,equal,""{name = e},bend right] & \uD_{\fd}\uD_{\fd}\uD_{\fd}(\mathcal{A}) \arrow[from=1-2,to=e,Rightarrow,"\phi"] \arrow[d,"\uD_{\fd}(\ev_{\mathcal{A}})"]\\
 & \uD_{\fd}(\mathcal{A})
\end{tikzcd}
\]
Define $\phi$ as the following DG-natural transformation.
\[
\begin{tikzcd}[row sep = large,column sep = large]
\uD_{\fd}\mathcal{A} \arrow[r,"\overline{\ev}_{\uD_{\fd}\mathcal{A}}"] \arrow[dddrr,"i^\mathcal{A}"] \arrow[dddddrrr,equal,bend right,""{name = e1}] & \uC_{\fd}\uC_{\fd} \uD_{\fd}(\mathcal{A}) \arrow[r,"\Res(i^{\uD_{\fd}(\mathcal{A})})"{yshift = 3pt}] \arrow[ddrr,equal,bend right,""{name=e2}] & \uC_{\fd}\uD_{\fd}\uD_{\fd}\mathcal{A}   \arrow[r,"Q^{\uD_{\fd}\uD_{\fd}(\mathcal{A})}" {yshift = 3pt}] \arrow[dr,equal,bend right,""{name=e3}] & \uD_{\fd}\uD_{\fd}\uD_{\fd}\mathcal{A} \arrow[d,"i^{\uD_{\fd}\uD_{\fd}(\mathcal{A})}"] \arrow[from=1-4,to =e3,Rightarrow,"\varepsilon^{\uD_{\fd}\uD_{\fd}(\mathcal{A})}"description,swap,shorten >= 1ex] \\ 
& &  & \uC_{\fd}\uD_{\fd}\uD_{\fd}\mathcal{A} \arrow[d,"\Res(Q^{\uD_{\fd}(\mathcal{A})})"] \\
& &  & \uC_{\fd}\uC_{\fd}\uD_{\fd}\mathcal{A} \arrow[d,"\Res\Res(i^{\mathcal{A}}))"] \\
& & \uC_{\fd}(\mathcal{A}) \arrow[from=4-3,to=e1,Rightarrow,"\varepsilon^{\mathcal{A}}"description] \arrow[r,"\overline{\ev}_{\uC_{\fd}(\mathcal{A})}"] \arrow[dr,equal]& \uC_{\fd}\uC_{\fd}\uC_{\fd}\mathcal{A} \arrow[d,"\Res(\overline{\ev}_{\mathcal{A}})"] \\
& & & \uC_{\fd}\mathcal{A} \arrow[d,"Q^{\mathcal{A}}"] \\
& & & \uD_{\fd}\mathcal{A} \arrow[from=e3,to=e2,Rightarrow,"\Res(\varepsilon^{\uD_{\fd}(\mathcal{A})})"description,shorten <= 2ex]
\end{tikzcd}
\]
The pentagon commutes by naturality of $\overline{\ev}$ and the unlabelled triangle is the triangle identity for the adjunction $\uC_{\fd}(-) \dashv \uC_{\fd}(-)$. Let $\beta$ denote the composite of the first three non-trivial DG-natural transformations in $\phi$, so that 
\[
\beta \colon \Res(\overline{\ev}_{\mathcal{A}}) \Res\Res(i^{\mathcal{A}}) \Res(Q^{\uD_{\fd}(\mathcal{A})}) i^{\uD_{\fd}\uD_{\fd}\mathcal{A}} Q^{\uD_{\fd}\uD_{\fd}\mathcal{A}}  \Res(i^{\uD_{\fd}(\mathcal{A})}) \overline{\ev}_{\uD_{\fd}\mathcal{A}} \to i^{\mathcal{A}}
\]
and $\phi = \varepsilon^{\mathcal{A}} Q^\mathcal{A}\beta$.
Then, given $M \in \uD_{\fd}(\mathcal{A})$ and $a \in \mathcal{A}$, $\beta_M(a)$ is the morphism in $\mathcal{D}^b(k)$ given by

\[
\begin{tikzcd}
(\Res(\overline{\ev}_{\mathcal{A}}) \Res\Res(i^{\mathcal{A}}) \Res(Q^{\uD_{\fd}(\mathcal{A})}) i^{\uD_{\fd}\uD_{\fd}\mathcal{A}} Q^{\uD_{\fd}\uD_{\fd}\mathcal{A}}  \Res(i^{\uD_{\fd}(\mathcal{A})}) \overline{\ev}_{\uD_{\fd}(\mathcal{A}),M})(a) \arrow[d,equal] \\
i^{\uD_{\fd}\uD_{\fd}(\mathcal{A})} Q^{\uD_{\fd}\uD_{\fd}(\mathcal{A})}( \Res(i^{\uD_{\fd}\mathcal{A}}) \overline{\ev}_{\uD_{\fd}(\mathcal{A}),M})(Q^{\uD_{\fd}(\mathcal{A})}(\Res(i^{\mathcal{A}}) \overline{\ev}_{\mathcal{A},a})) \arrow[d,"\varepsilon^{\uD_{\fd}\uD_{\fd}(\mathcal{A})}_{\Res(i^{\uD_{\fd}\mathcal{A}}) \overline{\ev}_{\uD_{\fd}\mathcal{A},M}} (Q^{\uD_{\fd}(\mathcal{A})}(\Res(i^{\mathcal{A}}) \overline{\ev}_{\mathcal{A},a}))"] \\
\Res(i^{\uD_{\fd}\mathcal{A}}) \overline{\ev}_{\uD_{\fd}\mathcal{A},M}(Q^{\uD_{\fd}(\mathcal{A})}(\Res(i^{\mathcal{A}}) \overline{\ev}_{\mathcal{A},a})) = (i^{\uD_{\fd}\mathcal{A}} Q^{\uD_{\fd}\mathcal{A}})(\Res(i^{\mathcal{A}}) \overline{\ev}_{\mathcal{A},a})(M)  \arrow[d,"\varepsilon^{\uD_{\fd}(\mathcal{A})}_{\Res(i^{\mathcal{A}})\overline{\ev}_{\mathcal{A},a}}(M)"] \\
\Res(i^{\mathcal{A}}) \overline{\ev}_{\mathcal{A},a}(M) = i^{\mathcal{A}}(M)(a)
\end{tikzcd}
\]
This is a quasi-isomorphism as both $\varepsilon$'s are pointwise quasi-isomorphisms. Therefore, $\beta_M$ is a quasi-isomorphism of $\mathcal{A}$-modules and so is $Q^{\mathcal{A}}(\beta_M)$. So the DG-natural transformation $\phi$ evaluated at $M$ is   
\[
\uD_{\fd}(\ev_{\mathcal{A}})\ev_{\uD_{\fd}(\mathcal{A})}(M) \xrightarrow[\sim]{Q^{\mathcal{A}}(\beta_M)} Q^{\mathcal{A}}i^{\mathcal{A}}(M) \xrightarrow[\sim]{\varepsilon^{\mathcal{A}}_M} M
\]
which is a quasi-isomorphism. Therefore, $\uD_{\fd}(\ev_{\mathcal{A}}) \ev_{\mathcal{D}_{\fd}(\mathcal{A})} = 1_{\uD_{\fd}(\mathcal{A})}$ as elements of $\oHom_{\Hmo}(\uD_{\fd}(\mathcal{A}),\uD_{\fd}(\mathcal{A}))$ by Lemma \ref{naturalitylemma}. Hence the triangle identities hold and there is an adjunction $\uD_{\fd}(-) \dashv \uD_{\fd}(-)$ with unit and counit $\ev$.
\end{proof}

\begin{remark}
The adjunction of Proposition \ref{adjunctionDfd}, shows that there are natural isomorphisms similar to those of Equation (1) of Section \ref{sectreflmon}. 
\end{remark}

In the next lemma we show that the isomorphism of Proposition \ref{Toenmodules} is natural.

\begin{lemma}\label{isofunctors}
There is an isomorphism $\uD_{\fd}(-) \simeq \hom_{\Hqe}(-,\uD^b(k))$ of functors  $\Hqe \to \Hqe^{op}$ and an isomorphism $\uD_{\fd}(-) \simeq \hom_{\Hmo}(-,k)$ of functors $\Hmo \to \Hmo^{op}$.
\end{lemma}

\begin{proof}
Recall from Proposition \ref{Toenmodules}, there is an isomorphism $\hom_{\Hqe}(\mathcal{A},\uD^b(k)) \simeq \uD_{\fd}(\mathcal{A})$ in $\Hqe$. It remains to check this isomorphism is natural. It is enough to prove the first statement and as in the proof of Proposition \ref{evdefn} we need only check naturality for DG-functors $F\colon \mathcal{A} \to \mathcal{B}$. For any cofibrant $X \in \Hqe$, there is an isomorphism by Lemma 6.2 in \cite{toe06}
\begin{equation} \label{Toen'sdude} 
\oHom_{\Hqe}(X, \uD_{\fd}(\mathcal{A})) \simeq \iso \Ho( \hom_{\dgcat}(X,\uC_{\fd}(\mathcal{A}))).
\end{equation}
using the same notation as the proof of Proposition \ref{Toenmodules}. We first show the following diagram commutes
\[
\begin{tikzcd}[row sep = large, column sep = large]
\oHom_{\Hqe}( X, \uD_{\fd}(\mathcal{B})) \arrow[d,"\uD_{\fd}(F) \circ"] \arrow[r,"\sim"] &  \iso \Ho(\hom_{\dgcat}(X,\uC_{\fd}(\mathcal{B}))) \arrow[d,"{\hom_{\dgcat}(X,\Res(F))}"] \\ 
\oHom_{\Hqe}( X, \uD_{\fd}(\mathcal{A})) \arrow[r,"\sim"] & \iso \Ho(\hom_{\dgcat}(X,\uC_{\fd}(\mathcal{A})))
\end{tikzcd}
\]
where the right vertical functor is well defined on the homotopy categories since $\Res(F)$ preserves weak equivalences in $\uC_{\fd}(\mathcal{B})$ and the weak equivalences in $\hom_{\dgcat}(X,\uC_{\fd}(\mathcal{B}))$ are defined pointwise. Any $f\colon X \to \uD_{\fd}(\mathcal{B}) \in \Hqe$ can be modelled as an actual DG-functor since $X$ is cofibrant. The proof of Lemma 6.2 in loc.\ cit.\ shows that $f$ is sent along the top right composition to  
\[
X \xrightarrow{f} \uD_{\fd}(\mathcal{B}) \xrightarrow{i^{\mathcal{A}}} \uC_{\fd}(\mathcal{B}) \xrightarrow{\Res(F)} \uC_{\fd}(\mathcal{A}).
\]
The image of $f$ along the bottom left composition is 
\[
X \xrightarrow{f} \uD_{\fd}(\mathcal{B}) \xrightarrow{i^{\mathcal{A}}} \uC_{\fd}(\mathcal{B}) \xrightarrow{\Res(F)} \uC_{\fd}(\mathcal{A}) \xrightarrow{Q^{\mathcal{A}}} \uD_{\fd}(\mathcal{A}) \xrightarrow{i^{\mathcal{A}}} \uC_{\fd}(\mathcal{A}).  
\]
Then, since $\varepsilon\colon i^{\mathcal{A}}Q^{\mathcal{A}} \to 1_{\uC_{\fd}(\mathcal{A})}$ is a pointwise quasi-isomorphism, these two objects are equal in $\iso \Ho(\hom_{\dgcat}(X,\uC_{\fd}(\mathcal{A}))$, as required. Next, note that the following diagram commutes
\[
\begin{tikzcd}[row sep = large, column sep = large]
\iso \Ho(\hom_{\dgcat}(X,\uC_{\fd}(\mathcal{B}))) \arrow[r,"\sim"] \arrow[d,"{\hom_{\dgcat}(X,\ \Res(F))}"]& \iso \Ho(\hom_{\dgcat}(X \otimes \mathcal{B},\uC_{\fd}(k)) \arrow[d,"{\hom_{\dgcat}(X \otimes F,\ \uC_{\fd}(k))}"] \\
\iso \Ho(\hom_{\dgcat}(X,\uC_{\fd}(\mathcal{A}))) \arrow[r,"\sim"] & \iso \Ho(\hom_{\dgcat}(X \otimes \mathcal{A},\uC^b(k)) 
\end{tikzcd}
\]
as it commutes at the non-derived level by naturality of the tensor-hom adjunction. Lemma 6.2 of \cite{toe06} gives us naturality of the isomorphism in Equation (\ref{Toen'sdude}) in $X$. Hence, the following diagram commutes.
\[
\begin{tikzcd}[row sep = large, column sep = large]
\iso \Ho(\hom_{\dgcat}(X \otimes \mathcal{B}, \mathcal{C}^b(k))) \arrow[d,"{\hom_{\dgcat}(X \otimes F,\ \mathcal{C}^b(k))}"]  \arrow[r,"\sim"] & \oHom_{\Hqe}(X \otimes \mathcal{B},\mathcal{D}^b(k)) \arrow[d," \circ (X \otimes F)"] \\
\iso \Ho (\hom_{\dgcat}(X \otimes \mathcal{A}, \uC^b(k))) \arrow[r,"\sim"] & \oHom_{\Hqe}(X \otimes \mathcal{A},\mathcal{D}^b(k))
\end{tikzcd}
\]
Finally, note that the square below commutes by definition of the functoriality of an internal hom.
\[
\begin{tikzcd}[row sep = large, column sep = large]
\oHom_{\Hqe}(X \otimes \mathcal{B},\mathcal{D}^b(k)) \arrow[d,"\circ (X \otimes F)"] \arrow[r,"\sim"] & \oHom_{\Hqe}(X,\hom_{\Hqe}(\mathcal{B},\mathcal{D}^b(k)))  \arrow[d,"{\hom_{\Hqe}(F,\mathcal{D}^b(k)) \circ }"] \\ 
\oHom_{\Hqe}(X \otimes \mathcal{A},\mathcal{D}^b(k)) \arrow[r,"\sim"] & \oHom_{\Hqe}(X,\hom_{\Hqe}(\mathcal{A},\mathcal{D}^b(k))) 
\end{tikzcd}
\]
Pasting the previous four diagrams together states that the image under Yoneda embedding of the following diagram commutes.
\[
\begin{tikzcd}[row sep = large, column sep = large]
\uD_{\fd}(\mathcal{B}) \arrow[d,"{\uD_{\fd}(F)}"] \arrow[r,"\sim"] & \hom_{\Hqe}(\mathcal{B},\uD_{\fd}(k)) \arrow[d,"{\hom_{\Hqe}(F,\mathcal{D}^b(k))}"] \\
\uD_{\fd}(\mathcal{\mathcal{A}}) \arrow[r,"{\sim}"] & \hom_{\Hqe}(\mathcal{B},\mathcal{D}^b(k)) 
\end{tikzcd} 
\]
Therefore, the diagram commutes in $\Hqe$.
\end{proof}

 Recall that $\Hmo$ is a closed symmetric monoidal category by Theorem \ref{toeinternalhmo} and Lemma \ref{isofunctors} states that the duality functor on $\Hmo$ is isomorphic to $\uD_{\fd}(-)$.

\begin{theorem} \label{samedef}
A DG-category $\mathcal{A}$ is reflexive if and only if it is a reflexive object in the closed symmetric monoidal category $\Hmo$.
\end{theorem}

\begin{proof}
As in Section \ref{sectreflmon}, the symmetric monoidal structure on $\Hmo$ implies there is an adjunction $\hom_{\Hmo}(-,k) \dashv \hom_{\Hmo}(-,k)$ with unit denoted $\eval$. By Proposition \ref{adjunctionDfd} there is an adjunction $\uD_{\fd}(-) \dashv \uD_{\fd}(-)$ with unit $\ev$ as constructed in Proposition \ref{evdefn}. By Lemma \ref{isofunctors}, there is an isomorphism of functors $\hom_{\Hmo}(-,k) \simeq \uD_{\fd}(-)$. So by uniqueness of adjunctions, the units must coincide up to isomorphism and the composition
\[
\begin{tikzcd}
\mathcal{A} \xrightarrow{\ev_{\mathcal{A}}} \uD_{\fd}\uD_{\fd}(\mathcal{A}) \simeq \hom_{\Hqe}(\hom_{\Hqe}(\mathcal{A},\mathcal{D}^b(k)),\mathcal{D}^b(k))
\end{tikzcd}
\]
is $\eval_{\mathcal{A}}$. Therefore $\ev_{\mathcal{A}}$ is an isomorphism if and only $\eval_{\mathcal{A}}$ is an isomorphism. 
\end{proof}

The following consequence will be of great use. It is an enhanced version of Corollary 3.16 in \cite{KS22}. Also note we only require $\mathcal{B}$ to be reflexive.

\begin{cor} \label{dudeequiv}
Suppose $\mathcal{B}$ is a reflexive (semireflexive) DG-category and $\mathcal{A}$ is any DG-category. Then the morphism in $\Hmo$ 
\[
\hom_{\Hmo}(\mathcal{A},\mathcal{B}) \to \hom_{\Hmo}(\uD_{\fd}(\mathcal{B}),\uD_{\fd}(\mathcal{A}))
\]
is a quasi-equivalence (quasi-fully faithful) and the map 
\[
\oHom_{\Hmo}(\mathcal{A},\mathcal{B}) \to \oHom_{\Hmo}(\uD_{\fd}(\mathcal{B}),\uD_{\fd}(\mathcal{A}))
\]
is an isomorphism (monomorphism).
\end{cor}

\begin{proof}

If $\mathcal{B}$ is reflexive, this follows immediately from Proposition \ref{reflequiv} and Theorem \ref{samedef}. If $\mathcal{B}$ is semireflexive, $\ev_{\mathcal{B}}$ is quasi-fully faithful. As in the proof of Theorem \ref{samedef}, $\eval_{\mathcal{B}}$ is quasi-fully faithful. Then by Corollary 6.6 in \cite{toe06}, the morphism $\hom_{\Hqe}(\mathcal{A},\eval_{\mathcal{B}})$ is also quasi-fully faithful. By Diagram \ref{evaladjointdiagram} in Remark \ref{enricheddual}, the morphism in the statement is quasi-fully faithful. The second result follows from the first by applying $H^0$ and restricting to isomorphism classes of objects as in Remark \ref{isoclassesobj}.
\end{proof}


\section{Applications to Hochschild Cohomology and Derived Picard Groups}
\label{HHreflsec}

In this section we apply Theorem \ref{samedef} to study Hochschild cohomology and Derived Picard groups. We recall that if $\mathcal{A}$ and $\mathcal{B}$ are DG-categories then by Remark \ref{enricheddual} and the Lemma \ref{isofunctors}, there are natural maps 
\begin{equation} \label{thedude}
\hom_{\Hmo}(\mathcal{A},\mathcal{B}) \to \hom_{\Hmo}(\uD_{\fd}(\mathcal{B}),\uD_{\fd}(\mathcal{A})).
\end{equation}
in $\Hmo$. By Remark \ref{isoclassesobj}, applying $H^0$ and restricting to isomorphism classes of objects gives the action of the functor $\uD_{\fd}(-)$ on morphisms in $\Hmo$.

\subsection{Hochschild Cohomology} 

\begin{theorem} \label{semireflhhiso}
If $\mathcal{A}$ is a semireflexive DG-category, then there is a quasi-isomorp\-hism of DG-algebras
\[
HH(\mathcal{A}) \simeq HH(\uD_{\fd}(\mathcal{A})).
\]

\end{theorem}

\begin{proof}
Since the morphism in Equation (\ref{thedude}) lifts the action of the functor $\uD_{\fd}(-)$, it sends $1_{\mathcal{A}}$ to $1_{\uD_{\fd}(\mathcal{A})}$. The morphism in Equation (4) is quasi-fully faithful by Corollary \ref{dudeequiv} and so it induces an isomorphism on $HH^\ast$ by Theorem \ref{toehh}.
\end{proof}

\begin{remark} We note that similar Hochschild cohomology isomorphisms are already known.
\begin{enumerate}
\item In Theorem 4.4.1 of \cite{LVdB05}, it is shown that $HH^\ast(\mathcal{A})$ is isomorphic to $HH^\ast(\mathcal{B})$ for any subcategory $\mathcal{A}^{op} \subseteq \mathcal{B} \subseteq \uD(\mathcal{A})$. In the non-proper case this does not apply to our situation since $\uD_{\fd}(\mathcal{A})$ does not always meet this condition. 

\item In \cite{kel19}, it was shown that Koszul duality produces isomorphisms on Hochschild cohomology.
\end{enumerate}
\end{remark}

\begin{ex}
The Hochschild homology of a DG-category $\mathcal{A}$ is the homology of the chain complex $\mathcal{A} \otimes^{\mathbb{L}}_{\mathcal{A}^e} \mathcal{A}$. The analogous version of Theorem \ref{semireflhhiso} for Hochschild homology does not hold. The algebra $k[x]/x^2$ is reflexive and $\uD_{\fd}(k[x]/x^2) \simeq \uD^{\perf}(A)$ where $A$ is the DG-algebra $k[t]$ with $\lvert t \rvert =1$. But one can compute that $HH_\ast(A)$ is non-zero in negative (homological) degrees. This cannot occur for the Hochschild homology of an algebra. 
\end{ex}

If $\mathcal{A}$ is proper, then it is semireflexive (Lemma 3.14 of \cite{KS22} or Proposition \ref{propersemiref}) and so by Theorem \ref{semireflhhiso}, there is a quasi-isomorphism $HH(\mathcal{A}) \simeq HH(\uD_{\fd}(\mathcal{A}))$. We note that this can be extended to any intermediary category.

\begin{theorem} \label{properhhinv}
Suppose $\mathcal{A}$ is a proper DG-category and $\mathcal{A}^{op} \subseteq \mathcal{B} \subseteq \uD_{\fd}(\mathcal{A})$. Then there is a quasi-isomorphism of DG-algebras $HH(\mathcal{A}) \simeq HH(\mathcal{B})$.
\end{theorem}

\begin{proof} We will use the interpretation of Hochschild cohomology of Theorem \ref{toehh}. Denote the inclusion $j\colon \mathcal{B} \hookrightarrow \uD_{\fd}(\mathcal{A})$. Consider the diagram
\[
\begin{tikzcd}[row sep = large]
\hom_{\Hqe}(\mathcal{A},\mathcal{A})  \arrow[r,hook] \arrow[dr,hook] &\hom_{\Hqe}(\mathcal{A},\uD_{\fd}\uD_{\fd}(\mathcal{A})) \arrow[d,"{\hom_{\Hqe}(\mathcal{A},\uD_{\fd}(j))}"]   \arrow[r,"\sim"] & \hom_{\Hqe}(\uD_{\fd}(\mathcal{A}),\uD_{\fd}(\mathcal{A})) \arrow[d,"{\hom_{\Hqe}(j,\uD_{\fd}(\mathcal{A})) \eqcolon j^\ast}"] \\
 & \hom_{\Hqe}(\mathcal{A},\uD_{\fd}(\mathcal{B})) \arrow[r,"\sim"] & \hom_{\Hqe}(\mathcal{B},\uD_{\fd}(\mathcal{A}))  \\
& & \hom_{\Hqe}(\mathcal{B},\mathcal{B}) \arrow[u,hook,"{j_{\ast} \coloneq  \hom_{\Hqe}(\mathcal{B},j)}"]  
\end{tikzcd}
\]
where the triangle is $\hom_{\Hqe}(\mathcal{A},-)$ applied to Lemma \ref{yonedalemma}. The square commutes by naturality of the adjunction. From this we see that $j^\ast$ is quasi-fully faithful when restricted to the image of $\hom_{\Hqe}(\mathcal{A},\mathcal{A})$. The long composite on the top is the map in Equation \ref{thedude} and so $1_{\uD_{\fd}(\mathcal{A})}$ is in the image of $\hom_{\Hqe}(\mathcal{A},\mathcal{A})$. It follows that the map  
\[
j^\ast \colon HH(\uD_{\fd}(\mathcal{A})) \xrightarrow{\sim} \hom_{\Hqe}(\mathcal{B},\uD_{\fd}(\mathcal{A}))(j^\ast(1_{\uD_{\fd}(\mathcal{A})}),j^\ast(1_{\uD_{\fd}(\mathcal{A})}))
\]
is a quasi-isomorphism. By Corollary 6.6 in \cite{toe06}, the map 
\[
j_\ast : HH^\ast(\mathcal{B}) \xrightarrow{\sim} \hom_{\Hqe}(\mathcal{B},\uD_{\fd}(\mathcal{A}))(j_\ast(1_{\mathcal{B}}),j_\ast(1_{\mathcal{B}}))
\]
is a quasi-isomorphism. Applying $H^0$ to $\hom_{\Hqe}(-,-)$ and taking isomorphism classes of objects gives $\oHom_{\Hqe}(-,-)$ and so $j^\ast(1_{\uD_{\fd}(\mathcal{A})}) \simeq j \simeq j_\ast(1_{\mathcal{B}}) \in H^0\hom_{\Hqe}(\mathcal{B},\uD_{\fd}(\mathcal{A}))$. Therefore, there is a quasi-isomorphism $HH(\uD_{\fd}(\mathcal{A})) \simeq HH^\ast(\mathcal{B})$. By Proposition \ref{propersemiref}, $\mathcal{A}$ is semireflexive so by Theorem \ref{semireflhhiso},
$HH(\mathcal{A}) \simeq HH(\uD_{\fd}(\mathcal{A}))$ and we are done. 
\end{proof}

\begin{remark}
Theorem \ref{properhhinv} also follows from Theorem 4.4.1 in \cite{LVdB05}.
\end{remark}

\begin{remark} \label{functorialityhh}
The ideas in the proof of Theorem \ref{properhhinv} should be compared to the limited functoriality of Hochschild cohomology. Keller showed that Hochschild cohomology is functorial with respect to quasi-fully faithful DG-functors in \cite{Kel03}. This can almost be seen from To\"en's description. If $F\colon \mathcal{A} \to \mathcal{B} \in \Hqe$, there is a diagram in $\Hqe$
\[
\hom_{\Hqe}(\mathcal{A},\mathcal{A}) \xrightarrow{F_\ast} \hom_{\Hqe}(\mathcal{A},\mathcal{B}) \xleftarrow{F^\ast} \hom_{\Hqe}(\mathcal{B},\mathcal{B})
\]
If $F$ is quasi-fully faithful, then by Corollary 6.6 in \cite{toe06}, $F_\ast$ is quasi-fully faithful. There are isomorphisms $F_\ast(1_{\mathcal{A}}) \simeq F \simeq F^\ast(1_{\mathcal{B}}) \in H^0\hom_{\Hqe}(\mathcal{A},\mathcal{B})$. Hence, there is a diagram of DG-algebras
\[
\hom_{\Hqe}(\mathcal{A},\mathcal{A})(1_{\mathcal{A}},1_{\mathcal{A}}) \xrightarrow{\sim} \hom_{\Hqe}(\mathcal{A},\mathcal{B})(F,F) \xleftarrow{} \hom_{\Hqe}(\mathcal{B},\mathcal{B})(1_{\mathcal{B}},1_{\mathcal{B}})
\]
which is a morphism from $HH(\mathcal{B}) \to HH(\mathcal{A})$ in $\Hqe$ using Theorem \ref{toehh}. As it depends on a choice of isomorphism $F_\ast(1_{\mathcal{A}}) \simeq F$, it is not clear this construction is functorial. This can presumably be fixed by working with the additional structure provided by viewing $\Hqe$ as a bicategory but we will not pursue this here. We note only that, whatever choices are made, it follows from the proof of Theorem \ref{properhhinv}, that for a proper DG-category $\mathcal{A}$, the map $HH(\uD_{\fd}(\mathcal{A})) \to HH(\mathcal{A})$ induced by $\mathcal{A}^{op} \hookrightarrow \uD_{\fd}(\mathcal{A})$ is a quasi-isomorphism.  
\end{remark}

\subsection{Derived Picard Groups}
\label{DPicsec}

The derived Picard group of a DG-category is the enhanced version of the group of triangulated autoequivalences of its perfect derived category. In Corollary 3.16 of \cite{KS22}, it was shown that $\D^{\perf}(\mathcal{A})$ and $\mathcal{D}_{\fd}(\mathcal{A})$ have isomorphic triangulated autoequivalence groups for reflexive DG-categories. An enhanced version of this theorem follows immediately from the monoidal characterisation of reflexive DG-categories. Recall that $\Hmo$ is the defined in Section \ref{duals}.

\begin{defn}
The derived Picard group $\DPic(\mathcal{A})$ of a DG-category $\mathcal{A}$ is the group of automorphisms of $\mathcal{A}$ in $\Hmo$. 
\end{defn}

\begin{remark} \label{Dfdderivedmap}
As functors preserve isomorphisms $\uD_{\fd}(-):\Hmo \to \Hmo^{op}$ restricts to  $\DPic(\mathcal{A}) \to \DPic(\uD_{\fd}(\mathcal{A}))^{op}$ for any $\mathcal{A},\mathcal{B} \in \Hmo$.
\end{remark}

\begin{theorem} \label{DPic}
If $\mathcal{A}$ is a reflexive (semireflexive) DG-category, the group homomorphism
\[
\DPic(\mathcal{A}) \to \DPic(\uD_{\fd}(\mathcal{A}))^{op}
\]
is an isomorphism (monomorphism).
\end{theorem}

\begin{proof}
If $\mathcal{A}$ is semireflexive, then by Corollary \ref{dudeequiv}, the induced map of Remark \ref{Dfdderivedmap} is a monomorphism. If $\mathcal{A}$ is reflexive and $F \in \DPic(\uD_{\fd}(\mathcal{A}))$, then by Corollary \ref{dudeequiv}, $F = \uD_{\fd}(G)$ for some $G \in \oHom_{\Hmo}(\mathcal{A},\mathcal{A})$. Since functors preserve isomorphisms, $\uD_{\fd}(F) \in \oHom_{\Hmo}(\uD_{\fd} \uD_{\fd}(\mathcal{A}),\uD_{\fd} \uD_{\fd}(\mathcal{A}))$ is an automorphism. As $\mathcal{A}$ is reflexive and by Proposition \ref{reflexivechar}, under the isomorphism $\mathcal{A} \simeq \uD_{\fd}\uD_{\fd}(\mathcal{A})$, $G$ corresponds to $\uD_{\fd}\uD_{\fd}(F)$. Therefore, $G \in \DPic(\mathcal{A})$, as required. 
\end{proof}

\section{Finite-dimensional DG-algebras}
\label{fdsec}

A DG-algebra is finite-dimensional if its underlying chain complex is finite-dimensional over $k$. This is a strictly stronger condition than properness although by \cite[Corollary 3.12]{RS20} (or the appendix of \cite{Goo24c} for a version over any field), every proper connective DG-algebra admits a finite-dimensional model. See \cite{Orl20}, \cite{Orl23}, \cite{goo23} for some background on finite-dimensional DG-algebras. In this section we apply the above results to finite-dimensional DG-algebras. \\

In \cite[Definition 2.2]{Orl20} Orlov introduced the semisimple quotient $A/J_+$ of a finite-dimensional DG-algebra $A$. Here $J_+$ is the DG-ideal  $J + d(J)$ where $J$ is the radical of the underlying algebra of $A$ and $d$ is the differential of $A$. This behaves like the radical of a finite-dimensional algebra in many ways. For example, $\mathcal{D}(A/J_+)$ is equivalent to a product of finite-dimensional division algebras by \cite[Proposition 2.6]{Orl20}

\begin{remark} \label{DcfneqDsf} For a finite-dimensional algebra the smallest thick subcategory (i.\,e. one closed under summands shifts and extensions) containing the simple modules is its bounded derived category. However for a finite-dimensional DG-algebra $A$, the question of when $A/J_+$ generates $\D_{\fd}(A)$ is more subtle. It is true for connective finite-dimensional DG-algebras by \cite[Proposition 3.9]{Orl23}. It is clear that $\thick(A/J_+)$ is the smallest thick subcategory containing all $A$-modules which admit finite-dimensional models. However \cite[Theorem 5.4]{efi18} gives an example of a module over a finite-dimensional DG-algebra which has finite-dimensional cohomology but does not admit a finite-dimensional model. This does not rule out the possibility that $A/J_+$ generates $\mathcal{D}_{\fd}(A)$ since it is not clear that every object in $\thick(A/J_+)$ admits a finite-dimensional model.  
\end{remark}

The Koszul dual of a finite-dimensional DG-algebra $A$ is defined as 
\[
A^! = \RHom_A(A/J_+,A/J_+)
\]
By Keller's tilting theory (e.\,g. \cite[Lemma 4.2]{Kel94}), there is a quasi-equivalence of DG-categories $\thick(A/J_+) \simeq \uD^{\perf}((A^!)^{op})$. We deduce the following result about Hochschild cohomology.

\begin{prop} \label{HHfdiso}
If $A$ is a finite-dimensional DG-algebra, then there is a quasi-isomorphism of DG-algebras
\[
HH(A) \simeq HH(A^!).
\]
\end{prop}

\begin{proof}
By \cite[Theorem 2.14]{goo23}, $A \in \thick(A/J_+)$. Since $\thick(A/J_+) \subseteq \mathcal{D}_{\fd}(A)$ we are in the situation of Theorem \ref{properhhinv} with $\mathcal{A} = A$ and $\mathcal{B} =  \thick(A/J_+)$. By Keller's tilting theory (e.\,g. \cite[Lemma 4.2]{Kel94}), $\thick(A/J_+) \simeq \uD^{\perf}(A^!)$. So by Theorem \ref{properhhinv}, $HH^\ast(A) \simeq HH^\ast(\uD^{\perf}(A^!))$. Then we are done since $HH^\ast(\uD^{\perf}(A^!))) \simeq HH^\ast(A^!)$ by Remark \ref{Hochsperf}.
\end{proof}

\begin{ex} \label{gradedualeg}  
If $A = k[x]/x^2$ with $\lvert x \rvert =1$ viewed as a DG-algebra with zero differential, then $A^! = k[[t]]$ and by Proposition \ref{HHfdiso} we deduce that $HH^\ast(k[[t]]) \cong HH^\ast(k[x]/x^2)$ which can be computed as $k[[t]]$ in degrees 0 and 1 and zero otherwise. 
\end{ex}

We have the following condition for a finite-dimensional DG-algebra to be reflexive.

\begin{prop} \label{refltest}
Suppose $A$ is a finite-dimensional DG-algebra and $\thick(A/J_+) = \mathcal{D}_{\cf}(A)$. Then $A$ is reflexive if and only if $\mathcal{D}_{\cf}(A^!) \subseteq \mathcal{D}^{\perf}(A^!)$. 
\end{prop}

\begin{proof}
Under the assumption that $\thick(A/J_+) = \mathcal{D}_{\cf}(A)$, Theorem 7.2 in \cite{goo23} states that the restricted Yoneda embedding of $\mathcal{D}_{\cf}(A)$ 
\[
Y\mid_{\mathcal{D}^{\perf}(A)} \colon \mathcal{D}^{\perf}(A) \hookrightarrow \mathcal{D}_{\cf}\mathcal{D}_{\cf}(A)^{op}
\]
is an equivalence if and only if $\mathcal{D}_{\cf}(A^!) \subseteq \mathcal{D}^{\perf}(A^!)$. By Lemma \ref{yonedalemma} this occurs exactly if $\ev_{\mathcal{D}^{\perf}(A)}$ is an equivalence. 
\end{proof}

\begin{remark}
The condition that $\mathcal{D}_{\cf}(A^!) \subseteq \mathcal{D}^{\perf}(A^!)$ states that $A^!$ is HFD closed in the sense of \cite{KS22}.  
\end{remark}

\begin{remark} \label{noethhfd}
If $R$ is a left Noetherian $k$-algebra of finite global dimension then every finite-dimensional module is finitely generated and so admits a finite resolution by finitely generated projectives. Hence, $\mathcal{D}_{\cf}(R) \subseteq \mathcal{D}^{\perf}(R)$. 
\end{remark}

\begin{ex} \label{powerseriesrefl}
Let $A$ be as in Example \ref{gradedualeg}. Then by Remark \ref{noethhfd}, $\mathcal{D}_{\fd}(A^!) \subseteq \mathcal{D}^{\perf}(A^!)$. By \cite[Proposition 4.6]{NK11}, $\thick(k) = \mathcal{D}_{\fd}(A)$ and so Proposition \ref{refltest} implies that $A$ is reflexive. Since $\uD_{\fd}(A)$ is Morita equivalent to $A^!$ we see that $A^!$ is also reflexive.
\end{ex}

\bibliographystyle{alpha}
\bibliography{biblio}

\end{document}